\documentclass[11pt,a4paper,oneside,reqno]{amsart}
\usepackage{amsmath}
\usepackage{amssymb}
\usepackage{amstext}
\usepackage{amsfonts}
\usepackage{amsthm}
\usepackage{ifthen}
\usepackage{xspace}
\usepackage{comment}
\usepackage{graphicx}
\usepackage{mathrsfs}
\usepackage{enumerate}
\usepackage{breakcites}
\usepackage{xcolor}

\allowdisplaybreaks[1]

\usepackage{hyperref}

\numberwithin{equation}{section}  

\newtheorem{punkt}{}[section]

\theoremstyle{plain}
\newtheorem{corollary}[punkt]{Corollary}
\newtheorem{lemma}[punkt]{Lemma}

\newtheorem{theorem}[punkt]{Theorem}

\newcommand{\eins}{\leavevmode\hbox{\small1\kern-3.8pt\normalsize1}}
\theoremstyle{definition}
\newtheorem{remark}[punkt]{Remark}

\newtheorem{example}[punkt]{Example}
\newtheorem{examples}[punkt]{Examples}

\newtheorem{definition}[punkt]{Definition}

\theoremstyle{plain}
\newtheorem*{corollary*}{Corollary}
\newtheorem*{lemma*}{Lemma}
\newtheorem*{proposition*}{Proposition}
\newtheorem*{theorem*}{Theorem}

\theoremstyle{definition}
\newtheorem*{remark*}{Remark}
\newtheorem*{remarks*}{Remarks}
\newtheorem*{example*}{Example}
\newtheorem*{examples*}{Examples}
\newtheorem*{problem*}{Problem}
\newtheorem*{problems*}{Problems}
\newtheorem*{question*}{Question}
\newtheorem*{questions*}{Questions}
\newtheorem*{definition*}{Definition}
\newtheorem*{conjecture*}{Conjecture}
\newtheorem*{assumption*}{Assumption}
\newtheorem*{assumptions*}{Assumptions}
\newtheorem*{construction*}{Construction}

\def\mycmplx{\mathbb{C}}

\def\mynat{\mathbb{N}}

\def\myreal{\mathbb{R}}

\def\U{\text{U}}

\def\mys{\mathcal{S}}

\def\ee{\mathbb{E}}		

\def\tr{\qopname\relax{no}{tr}}

\def\eg{e.g.\@\xspace}
\def\ie{i.e.\@\xspace}

\def\GL{\qopname\relax{no}{GL}}

\def\fourier{\mathcal{F}}

\def\mellin{\mathcal{M}}
\def\hankel{\mathcal{H}}

\begin{document}

\renewcommand{\thefootnote}{}

\title[Polynomial Ensembles \& P\'olya Frequency Functions]{Polynomial Ensembles and P\'olya Frequency Functions}
\author[Yanik-Pascal F\"orster, Mario Kieburg and Holger K\"osters]{Yanik-Pascal F\"orster$^{1,2,6}$, Mario Kieburg$^{1,3,7}$ and Holger K\"osters$^{4,5,8,*}$}
%

\date{\today}

\begin{abstract}
We study several kinds of polynomial ensembles of derivative type which we 
propose to call P\'olya ensembles. These ensembles are defined on the spaces 
of complex square, complex rectangular, Hermitian, Hermitian anti-symmetric 
and Hermitian anti-self-dual matrices, and they have nice closure properties 
under the multiplicative convolution for the first class and under the 
additive convolution for the other classes. The cases of complex square 
matrices and Hermitian matrices were already studied in former works. One of 
our goals is to unify and generalize the ideas to the other classes of 
matrices. Here we consider convolutions within the same class of P\'olya 
ensembles as well as convolutions with the more general class of polynomial 
ensembles. Moreover, we derive some general identities for group integrals 
similar to the Harish-Chandra-Itzykson-Zuber integral, and we relate P\'olya 
ensembles to P\'olya frequency functions. For illustration we give a number 
of explicit examples for our results.
\end{abstract}



\maketitle{}

\footnotetext{*) Corresponding Author}
\footnotetext{1) Department of Physics, Bielefeld University, Postfach 100131,
 D-33501 Bielefeld, Germany}
\footnotetext{2) Faculty of Natural \& Mathematical Sciences, King's College London, Strand, London, WC2R~2LS, United Kingdom}
\footnotetext{3) University of Melbourne, School of Mathematics and Statistics,
 813 Swanton Street, Parkville, Melbourne VIC 3010, Australia}
\footnotetext{4) Department of Mathematics, Bielefeld University,
 Postfach 100131, D-33501 Bielefeld, Germany}
\footnotetext{5) Institute of Mathematics, University of Rostock,
 D-18051 Rostock, Germany}
\footnotetext{6) \emph{Email address}: yanik-pascal.foerster@kcl.ac.uk}
\footnotetext{7) \emph{Email address}: m.kieburg@unimelb.edu.au}
\footnotetext{8) \emph{Email address}: holger.koesters@uni-rostock.de}


\section{Introduction}
\label{sec:Introduction}

In 2012 it was observed that the spectral statistics of certain products of independent random matrices may be calculated explicitly at finite matrix dimension \cite{AB:2012,AKW:2013}. This breakthrough did not only lead to a new and fast develop\-ment on products of random matrices, see e.g.\@ \cite{AI:2015} for a review, but also on related topics like sums of random matrices~\cite{CKW:2015,KR:2016}. Originally the whole development started with products of independent Ginibre and Jacobi (truncated unitary) matrices due to their simplicity and their field of applications, e.g. for the local spectral statistics at finite and infinite matrix dimension, see~\cite{AB:2012,AKW:2013,AS:2013,AIK:2013,IK:2013,ARRS:2013,ABKN:2014,Forrester:2014,KS:2014,KZ:2014,LWZ:2014,KKS:2015}. While the results and proofs for these ensembles relied on the particular form of the ensembles, it soon became clear that there is some common structure in the background which leads to unified proofs as well as to further generalizations. For instance, in \cite{KS:2014}, the notion of a \emph{polynomial ensemble} was introduced, and it was shown that this yields a convenient framework to investigate the multiplication of a random matrix by an independent Ginibre or Jacobi (truncated unitary) matrix \cite{KS:2014,KKS:2015}. Motivated by these findings, the concept of a \emph{polynomial ensemble of derivative type} was introduced in \cite{KK:2016a,KK:2016b}. These matrix ensembles have several important properties. First, they are isotropic (also called bi-unitarily invariant or rotationally invariant), and second, they depend on a single one-point weight, cf. Definition~\ref{def:PolyaEns}\,(2) below. Third, they generalize the results about the multiplication by Ginibre and Jacobi matrices to a larger subclass of polynomial ensembles. Interestingly, this concept does not only lead to a unified perspective on many of the preceding results, but it also includes several further prominent examples of complex (non-Hermitian) random matrix ensembles.

Shortly after the appearance of these results, it was noted in~\cite{KR:2016} that the notion of a \emph{polynomial ensemble of derivative type} may also be adapted to investigate sums of independent Hermitian random matrices. Our main aim is to extend these~results to further symmetry classes of random matrices, namely complex rectangular matrices, Hermitian anti-symmetric matrices, and Hermitian anti-self-dual matrices. As we shall see, all these classes can be dealt with in a similar way by using the appropriate multivariate transforms from harmonic analysis. For the additive convolution on the space of Hermitian matrices and the multiplicative convolution on the space of complex square matrices, one needed the matrix-variate Fourier transform and the spherical transform, i.e.\@ the multivariate counterparts of the univariate Fourier and Mellin transform, respectively.
We will generalize this idea by identifying the appropriate matrix version of the Hankel transform \cite[Chapter 10.22(v)]{NIST}, which is intimately related to the addition of isotropically distributed random vectors, for the above-mentioned classes of matrices. This is our first major goal.

Note that all five kinds of convolutions considered in the present work are ensembles of Dyson index $\beta=2$ because they are related to either compact Lie algebras or complex matrices. The reason why one can easily deal with all five classes in a similar way is the knowledge of certain group integrals involved, namely the Harish-Chandra-Itzykson-Zuber integral~\cite{HC,IZ}, the Gelfand-Na\u{\i}mark integral~\cite{GelNai}, and the Berezin-Karpelevich integral~\cite{BK,GW}, which have essentially the same structure. 
For random matrix ensembles corresponding to the Dyson indices $\beta=1$ or $\beta=4$,
explicit results for such group integrals are only known for very small matrix dimensions but not in general.

We have two further goals. The second goal is to explore the connection of the class of polynomial ensembles of derivative type to the class of \emph{P\'olya frequency functions}~\cite{Polya:1913,Polya:1915}, a notion from classical analysis \cite{Karlin:1968}. There exists some related work in this direction in representation theory and ergodic theory \cite{Schoen:1951,Karlin:1968,Pickrell:1991,OV:1996,Faraut2006,Bufetov}, 
but it seems that this connection has not been explored yet from the viewpoint of random matrix theory, \ie as regards the associated singular value and eigenvalue distributions at finite matrix dimension. The relation between P\'olya frequency functions (with certain differentiability and integrability properties) and poly\-nomial ensembles of derivative type will be essentially bijective for the cases of complex square matrices (with multiplication) and Hermitian random matrices (with addition), cf. Theorem \ref{thm:rel-Polya}, while we will only establish a certain injective relationship for the~other matrix classes under consideration, cf. Theorem \ref{thm:rel-Polya.b}. Anyway, we suggest to call these \emph{polynomial ensembles of derivative type} by the shorter name \emph{P\'olya ensembles}.

Our third goal is to investigate some generalizations of the above-mentioned group integrals, see~also~\cite{GR:1989,O:2004,HO:2006}. We obtain a number of new and highly non-trivial examples of such identities which  arise naturally from our approach. Incidentally, these identities for group integrals also play a central role with respect to our second goal.

The present work is organized as follows. In Sec.~\ref{sec:mainresults} we introduce
the necessary notation and state our main results. In Sec.~\ref{sec:transforms},
we introduce the univariate and multi\-variate transforms from harmonic analysis, which are central to our approach. Sec.~\ref{sec:proofs} is devoted to the proofs of our main results. We discuss and summarize our findings in Sec.~\ref{sec:conclusio}.

\section{Notation and Main Results} 
\label{sec:mainresults}

\subsection{Matrix Spaces}
\label{sec:matrixspaces}

We introduce the relevant matrix spaces, the corresponding group actions
as well as the sets of probability measures invariant under these actions,
and we recall the relations of these probability measures 
to the induced probability measures on the eigenvalues or singular values.

We adapt the notation from our previous works~\cite{KK:2016a,KK:2016b} 
to our current purposes. 
Let ${\rm O}(n)$, ${\rm U}(n)$ and ${\rm USp}(2n)$ denote 
the classical groups of orthogonal, unitary and unitary symplectic matrices,
and let ${\rm o}(n)$, ${\rm u}(n)$ and ${\rm usp}(2n)$ denote
the associated classical Lie algebras.

We are interested in the following matrix spaces,
with $n \in \mynat$:
\begin{itemize}
\item[(1)] $G :={\rm GL}(n,\mathbb{C})$ is the group 
of invertible complex $n \times n$ matrices,
endowed with the action of the group
$\hat{K} := {\rm U}(n) \times {\rm U}(n)$
via $(k_1,k_2).g := k_1 g k_2^*$.
\item[(2)] $H_2 := \imath {\rm u}(n) = {\rm Herm}(n)$ is the linear space
of Hermitian $n \times n$ matrices,
endowed with the action of the group
$K_2 := {\rm U}(n)$
via $k.y := k y k^*$.
\item[(3a)] 
For fixed $\nu \in \mynat_0$,
${\rm Mat}_{\mathbb{C}}(n,n+\nu)$ is the linear space
of complex $n\times (n+\nu)$ rectangular matrices,
endowed with the action of the group
${\rm U}(n) \times {\rm U}(n+\nu)$
via $(k_1,k_2).y := k_1 y k_2^*$.
We identify an element $y \in {\rm Mat}_{\mathbb{C}}(n,n+\nu)$ with a chiral Hermitian matrix
\begin{equation}\label{embedding.1}\left[\begin{array}{cc} 0 & y \\ y^* & 0 \end{array}\right] \in {\rm Herm}(2n+\nu)\end{equation}
and an element $k = (k_1,k_2) \in {\rm U}(n) \times {\rm U}(n+\nu)$ with a block matrix
\begin{equation}\label{embedding.2}\left[\begin{array}{cc} k_1 & 0 \\ 0 & k_2 \end{array}\right] \in {\rm U}(2n +\nu)\,.\end{equation}
With these identifications, the group action
may be written as $k.y = k y k^*$.
Write $M_\nu$ and $\hat{K}_\nu$ for the spaces of matrices
in \eqref{embedding.1} and \eqref{embedding.2}.
\item[(3b)] $H_1 := \imath {\rm o}(2n)$ or $H_1 := \imath {\rm o}(2n+1)$
and $H_4 := \imath {\rm usp}(2n)$ are the linear~spaces
of Hermitian anti-symmetric and Hermitian anti-self-dual matrices,
en\-dowed with the action of the groups 
$K_1 := {\rm O}(2n)$ or $K_1 := {\rm O}(2n+1)$
and $K_4 := {\rm USp}(2n)$ via $k.y := k y k^*$.
\end{itemize}
Moreover, we also need the linear space
of real diagonal $n \times n$ matrices, 
$D \simeq \mathbb{R}^n$,
and the group of positive-definite diagonal $n \times n$ matrices,
$A \simeq \mathbb{R}_{+}^n$,
each endowed with the natural action of the symmetric group $\mathbb{S}$
of all permutations of order $n$ on the diagonal elements.

When it is possible to consider several of these cases simultaneously,
we write $M$ for the matrix space and $K$ for the associated group.
Note that the dependence on $n$ is implicit.
When the need arises to make it explicit, we write 
$M(n)$ instead of $M$, $K(n)$ instead of $K$, etc.

In (2) and (3b) we have used the index $\beta=1,2,4$ to underline 
the connection to the field of real, complex and quaternion numbers, respectively;
yet, it has to be distinguished from the level repulsion
that corresponds to Dyson index $2$ in all cases.
The multiplication by $\imath$ is convenient since it leads to matrices 
with real eigenvalues.
However, it is clear that our results can easily be translated into results 
for~real anti-symmetric, anti-Hermitian and anti-Hermitian 
anti-self-dual (or~quaternion anti-Hermitian) matrices, respectively.
We treat the case $\beta = 2$ separately because it turns out 
to be simpler than the other cases.

As reference measures on the matrices spaces $G$, $M_\nu$, $H_\beta$, $A$ and $D$,
we~use the flat Lebesgue measures on the linearly independent matrix entries,
typically denoted by $dg$ for $G$, by $dy$ for $M_\nu$ and $H_\beta$,
and by $da$ for $A$ and $D$.
For the groups $\hat{K}$, $\hat{K}_\nu$, $K_\beta$ as well as ${\rm U}(n)$,
we take the normalized Haar measures, always denoted by $d^* k$.
Occasionally, we also use the Haar measure $d^*g = dg/\det (gg^*)^n$
on $G = \GL(n,\mycmplx)$, with $g^*$ the Hermitian adjoint of $g$.

We always ignore sets of measure zero.
Thus, the spaces $G$ and $M_0$ are essentially the same.
However, we prefer to use different notations to reflect 
the different group operations on these spaces, 
namely matrix multiplication on $G$ and matrix addition on $M_0$.
Thus, when we speak about random matrices on $G$ or on $M_0$, 
we will be interested in their products and sums,
respectively.

\subsection{Matrix Densities and Spectral Densities}
\label{sec:matrixdensities}

Let $M$ and $K$ be as in Sub\-section \ref{sec:matrixspaces}.
As functions on $M$, we consider integrable functions invariant 
under the action of~$K$,
i.e.
\begin{equation}\label{group-inv-func}
L^{1,K}(M)=\left\{\left.f_{M}\in L^{1}(M)\right| f_{M}(k.m)=f_{M}(m)\ \forall m\in M,\,k\in K \right\} \,.
\end{equation}
In the following, we indicate the space on which the function (or density) 
is defined by a subscript, e.g.\@ $f_G,f_{M_\nu},f_{H_\beta},\hdots$
The invariance of the functions is called \emph{$K$-invariance} 
with respect to the respective group $K=\hat{K},\hat{K}_\nu,K_\beta,\mathbb{S}$. 
Note that this $K$-invariance 
amounts to \emph{bi-unitary invariance}
for the spaces ${\rm GL}(n,\mycmplx)$ and ${\rm Mat}_\mycmplx(n,n+\nu)$,
to \emph{conjugation invariance} for the spaces $H_\beta$,
and to \emph{permutation invariance} or \emph{symmetry}
for the spaces $D$ and $A$.
The subset of all probability densities in the set $L^{1,K}(M)$
will be denoted by $L_{\rm Prob}^{1,K}(M)$.

For each matrix space $M$ as in (1) -- (3), the $K$-invariant probability densities 
are in one-to-one correspondence with the induced \emph{spectral densities}, 
i.e.\@ the induced symmetric probability densities of the eigenvalues (for $M = H_2$) 
or the (non-zero) \linebreak squared singular values (for $M = G,M_\nu,H_1,H_4$). 
Let us describe these correspondences by bijective mappings $\mathcal{I}_M$ 
which associate to each $K$-invariant matrix density the induced spectral density. 
It turns out convenient to consider these mappings not only 
on the sets $L^{1,K}_{\text{Prob}}(M)$ of all $K$-invariant probability densities, 
but also, via linear extension, on the larger sets $L^{1,K}(M)$ 
of all $K$-invariant integrable functions.
Let the \emph{Vandermonde determinant} be defined by
\begin{equation}\label{Vandermonde}
\Delta_n(a)=\prod_{1\leq b<c\leq n}(a_c-a_b)=\det[a_l^{k-1}]_{l,k=1,\ldots,n},
\quad a_1,\hdots,a_n \in \myreal \,.
\end{equation}
Then the mappings $\mathcal{I}_M$ are as follows:
\begin{itemize}
 \item[(1)] 
Almost every matrix $y \in H_2$
has $n$ distinct eigenvalues $a_1,\hdots,a_n \in \myreal$,
and $\mathcal{I}_{H_2}:\, L^{1,K_2}(H_2)\rightarrow L^{1,{\mathbb{S}}}(D)$ with
 \begin{equation}\label{I-H2}
(\mathcal{I}_{H_2} f_{H_2})(a)= C_{H_2} \, f_{H_2}(a) \, \Delta_n^2(a) =:f_D(a), \quad a \in D \,.
 \end{equation}
\item[(2)]
Almost every matrix $g \in G$ 
has $n$ distinct squared singular values \linebreak $a_1,\hdots,a_n > 0$, 
and $\mathcal{I}_G : L^{1,\hat{K}}(G) \to L^{1,\mathbb{S}}(A)$ with
 \begin{equation}\label{I-G}
(\mathcal{I}_G f_G)(a) = C_G \, f_G(\sqrt{a}) \, \Delta_n^2(a) =:f_A(a) \,,\quad a \in A \,.
\end{equation}
\item[(3)]
Almost every matrix $y \in M \in \{ M_\nu, H_1, H_4 \}$ 
has $n$ distinct squared singular values \mbox{$a_1,\hdots,a_n > 0$},
which are the squares of the non-zero eigenvalues
$\pm \sqrt{a_1},\hdots,\pm \sqrt{a_n}$, 
and $\mathcal{I}_{M} : L^{1,K}(M) \to L^{1,\mathbb{S}}(A)$ with
 \begin{equation}\label{I-M}
(\mathcal{I}_{M} f_{M})(a) = C_M \, (\det a)^\nu \, f_{M}\left(\iota_{M}(a)\right) \Delta_n^2(a) =:f_A(a) \,,\quad a \in A \,,
\end{equation}
with $\nu \in \mynat_0 \cup \{ \pm \tfrac12 \}$ and $\iota_M$ as defined below.
\end{itemize}
Here
\begin{align*}
&\text{$\nu \in \mynat$ is given} 
&&\text{and}
&&\iota_{M}(a) := \left[\begin{array}{cc} 0 & \sqrt{a}\,\Pi_{n,n+\nu} \\ \Pi_{n,n+\nu}^*\,\sqrt{a} & 0 \end{array}\right] 
\mskip-80mu&&\mskip+80mu\text{for $M = M_\nu$\,,} \\[+10pt]
&\nu := -\tfrac12 
&&\text{and}
&&\iota_M(a) := \sqrt{a} \otimes \tau_2
&&\text{for $M = H_1 = \imath {\rm o}(2n)$\,,} \\
&\nu := +\tfrac12 
&&\text{and}
&&\iota_M(a) := \left[\begin{array}{cc} \sqrt{a} \otimes \tau_2 & 0 \\ 0 & 0 \end{array}\right]
&&\text{for $M = H_1 = \imath {\rm o}(2n\!+\!1)$\,,} \\
&\nu := +\tfrac12 
&&\text{and}
&&\iota_M(a) := \sqrt{a} \otimes \tau_3
&&\text{for $M = H_4 = \imath {\rm usp}(2n)$\,,}
\end{align*}
where $\Pi_{j,k}$ ($j \le k$) is the projection onto the first $j$ out of $k$ rows,
the square root $\sqrt{a}$ is taken component-wise,
$$
\tau_2 := \left[\begin{array}{cc} 0 & -\imath \\ \imath & 0 \end{array}\right]
\quad\text{and}\quad
\tau_3 := \left[\begin{array}{cc} 1 & 0 \\ 0 & -1 \end{array}\right]
$$
are the second and third Pauli matrix, and
$$
x \otimes y := \left( \begin{array}{ccc}
	x_{11} y & x_{12} y & \cdots \\
	x_{21} y & x_{22} y & \cdots \\
	\vdots & \vdots & \\
\end{array} \right)
$$
is the Kronecker product of two matrices $x = (x_{ij})$ and $y = (y_{ij})$.
The constants in Eqs.~\eqref{I-H2} -- \eqref{I-M} are given by
\begin{equation}\label{const}
C_{H_2}= \frac{1}{n!}\prod_{j=0}^{n-1}\frac{\pi^j}{j!} \,,\
C_{G}=C_{n,0}^* \,,\
C_{M_\nu}=C_{n,\nu}^* \,,\
C_{H_1}=C_{n,\nu}^* \,,\
C_{H_4}=\frac{C_{n,\nu}^*}{2^{n(n-1)}} \,,
\end{equation}
where
\begin{equation}\label{const*}
C_{n,\nu}^*= \frac{1}{n!}\prod_{j=0}^{n-1}\frac{\pi^{2j+\nu+1}}{\Gamma[j+\nu+1]j!}
\end{equation}
and $\Gamma$ denotes the Gamma function.
These constants may be found using the techniques
from Chapter~3 in \cite{Hua}, for example.

Let us emphasize that the $K$-invariance of the matrix functions $f_M$
is crucial for the bijectivity of the mappings $\mathcal{I}_M$.
Furthermore, the mappings $\mathcal{I}_M$ remain bijective
when restricted to probability densities.
Finally, the reader familiar with \cite{KK:2016a} should be warned 
that the names of the operators $\mathcal{I}_M$
follow a different logic than in \cite{KK:2016a}.

In particular, the $K$-invariant functions on $M$ correspond 
to symmetric functions of $n$ eigenvalues or squared singular values.
This observation will also be important in Section~\ref{sec:transforms},
where we introduce the multivariate transforms central to our approach.

\subsection{Convolutions}
\label{sec:Convolutions}

On the linear matrix spaces $M = H_1,H_2,H_4,M_\nu$,
the additive convolution is defined by
\begin{align}
\label{addconv}
(f_M \ast h_M)(y) =\int_M f_M(y') h_M(y-y') \, dy' \qquad (y \in M)
\end{align}
for $f_M,h_M \in L^1(M)$,
and on the matrix group $M = G$, 
the multiplicative con\-volution is defined by 
\begin{align}
\label{multconv}
(f_M \circledast h_M)(g)&=\int_M f_M(g') h_M((g')^{-1} g) \, d^*g' \qquad (g \in G)
\end{align}
for $f_M,h_M \in L^1(M)$, where $d^*g'=dg'/\det(g'{g'}^*)^n$ denotes the Haar measure on~$G$.
In terms of random matrices, these convolutions describe the density
of the sum or product of two independent random matrices $X_1$ and $X_2$
with the densities $f_M$ and $h_M$. For instance, in the additive case, 
we have
\begin{align*}
   \ee(\varphi(X_1+X_2)) 
&= \int_M \int_M \varphi(y_1+y_2) f_M(y_1) h_M(y_2) \, dy_2 \, dy_1 \\
&= \int_M \int_M \varphi(y) f_M(y_1) h_M(y - y_1) \, dy \, dy_1
\end{align*}
for all non-negative measurable functions $\varphi$,
where we have made the change of variables $y=y_1+y_2$, $dy=dy_2$.
This shows that $X_1+X_2$ has the density \eqref{addconv}.
In the multiplicative case, the claim follows by a similar calculation
using the change of variables $g=g_1g_2$, $d^*g = d^* g_2$.

When $f_M$ and $h_M$ are additionally $K$-invariant, their convolution is also $K$-invariant. 
It is then natural to consider the convolution at the level of the spectral densities.
We denote these \emph{induced convolutions} by
\begin{equation}
\label{conv-G}
f_A\circledast h_A := \mathcal{I}_{G}(\mathcal{I}_{G}^{-1} f_A \circledast \mathcal{I}_{G}^{-1} h_A)
\qquad (f_A,h_A\in L^{1,\mathbb{S}}(A))
\end{equation}
for $M = G$, by
\begin{equation}
\label{conv-H2}
f_D \ast h_D := \mathcal{I}_{H_2}(\mathcal{I}_{H_2}^{-1} f_D \ast \mathcal{I}_{H_2}^{-1} h_D)
\qquad (f_D,h_D\in L^{1,\mathbb{S}}(D))
\end{equation}
for $M = H_2$, and by
\begin{equation}
\label{conv-M}
f_{A}\ast_{\nu}h_{A} := \linebreak \mathcal{I}_{M}(\mathcal{I}_{M}^{-1} f_A \ast \mathcal{I}_{M}^{-1} h_A) \,, 		
\qquad (f_A,h_A\in L^{1,\mathbb{S}}(A))
\end{equation}
for $M \in \{ M_\nu,H_1,H_4 \}$, with $\nu \in \mynat \cup \{ \pm \tfrac12 \}$
defined as in \eqref{I-M}.

In Eq.~\eqref{conv-M}, the problem arises that for $\nu = +\tfrac12$,
there are two different choices for $M$, namely $M = \imath{\rm o}(2n+1)$ and $M = \imath{\rm usp}(2n)$.
However, as we will see in Section~\ref{sec:transforms}, both choices lead to the same convolution
$\ast_{1/2}$ on the space $L^{1,\mathbb{S}}(A)$.

We will use the induced convolutions $\circledast$, $\ast$ and $\ast_\nu$ primarily for $n = 1$,
where they reduce to convolutions on the spaces $L^1(\myreal_+)$, $L^1(\myreal)$ and $L^1(\myreal_+)$,
respectively. While the first two convolutions are simply 
the ordinary multiplicative and additive convolutions
on $\myreal_+$ and $\myreal$, respectively, 
the third convolution is closely related to the \emph{Hankel convolution};
see the comments at the beginning of Section~\ref{sec:transforms}.

\begin{remark}[Products of Rectangular Matrices]
As already mentioned, the multi\-plicative convolution on $G$ may be used
to study products of independent bi-invariant random square matrices.
One could also consider products of independent bi-invariant random rectangular matrices.
However such products can always be traced back to products of independent bi-invariant 
random square matrices with modified but related densities, see \cite{IK:2013}. 
This statement is not true for sums of $K$-invariant rectangular random matrices. 

For instance, given two independent bi-invariant random matrices $g_1\in \mathbb{C}^{n\times(n+\nu_1)}$ and $g_2\in \mathbb{C}^{(n+\nu_1)\times(n+\nu_2)}$ with $\nu_1,\nu_2 \in \mynat_0$,
we may write
$$
g_1 = \hat{g}_1 \Pi_{n,n+\nu_1} k_1 \quad\text{and}\quad \Pi_{n,n+\nu_1} k_1 g_2 = \hat{g}_2 \Pi_{n,n+\nu_2} k_2
$$
with $\hat{g}_1,\hat{g}_2 \in \mathbb{C}^{n \times n}$ bi-invariant,
$k_1 \in {\rm U}(n+\nu_1),k_2 \in {\rm U}(n+\nu_2)$ Haar distributed,
and all of them independent, to obtain a representation
$g_1g_2=\hat{g}_1\hat{g}_2 \Pi_{n,n+\nu_2} k_2$
for the product.

In contrast to that, given two independent bi-invariant random matrices $y_1,y_2 \linebreak[1] \in \mathbb{C}^{n\times(n+\nu)}$,
it does not seem possible in general to find a representation 
$y_1+y_2=(\hat{y}_1+\hat{y}_2) \Pi_{n,n+\nu} k$
with $\hat{y}_1,\hat{y}_2 \in \mathbb{C}^{n \times n}$ bi-invariant, 
$k \in {\rm U}(n+\nu)$ Haar distributed, and all of them independent.
\end{remark}

\subsection{Polynomial Ensembles}
\label{sec:PolEns}

Polynomial ensembles were introduced by \linebreak Kuijlaars and co-authors~\cite{KS:2014}. They have a simple algebraic structure which occurs in many prominent random matrix ensembles. In a previous work \cite{KK:2016b}, we identified a subset of these polynomial ensembles which is closed under the multiplicative convolution \eqref{conv-G} on $G$. This behaviour is in general not true for general polynomial ensembles. After that, a subset of polynomial ensembles with similar properties was investigated for the additive convolution \eqref{conv-H2} on $H_2$ in \cite{KR:2016}.

The main purpose of this subsection is to introduce similar subsets for the additive convolution \eqref{conv-M} on the spaces $M_\nu$, $H_1$ and $H_4$ 
of rectangular matrices, Hermitian anti-symmetric matrices, and Hermitian anti-self-dual matrices. For comparison and for later use, we also briefly describe the existing results for the other classes. Thus, we define three  subsets of polynomial ensembles.  

For an interval $\mathbb{I} \subset \myreal$ and a (measurable) subset $\mathcal{N} \subset \myreal$, let
\begin{align*}
 L^{1}_{\mathbb{I}}(\mathcal{N})=\biggl\{f\in L^1(\mathcal{N}) \,\biggl|\, \text{for all $\kappa\in\mathbb{I}$}:\ \int_{\mathcal{N}}|x|^{\kappa-1}|f(x)|\,dx<\infty\biggl\} \,.
\end{align*}

\pagebreak[2]

\begin{definition}[Polynomial ensembles]\label{def:PolEns} \
\begin{enumerate}[(i)]
\item
See~\cite{KS:2014}.
Let $n \in \mynat$ and $\mathcal{N}\subset \myreal$ be a subset.
A probability measure $\mu$ on $\mathcal{N}^n$ is called the \emph{polynomial ensemble} on $\mathcal{N}^n$
associated with the one-point weights $w_1,\ldots,w_n\in L^{1}_{[1,n]}(\mathcal{N})$ 
if it has a Lebesgue density of the form
\begin{equation}\label{jpdf-pol-ens}
	p(a)=C_n[w]\Delta_n(a)\det[w_b(a_c)]_{b,c=1,\ldots,n} \ge 0 \,,\quad a \in \mathcal{N}^n \,,
\end{equation}
where $C_n[w] > 0$ is the normalization constant. 
When $\mathcal{N}= \myreal$ or $\mathcal{N}= \myreal_+$, we also call $\mu$
a polynomial ensemble on $D$ or $A$, in line with 
our identifications $D \simeq \myreal^n$ and $A \simeq \myreal_+^n$.
\item
For $M \in \{ G,H_2,M_\nu,H_1,H_4 \}$,
a probability measure on $M$
with a density $f_M \in L^{1,K}_{\rm Prob}(M)$
is called a \emph{polynomial ensemble on $M$} 
if $\mathcal{I}_{M} f_{M}$
is the density of a polynomial ensemble on $D$ or $A$
(i.e.\@ if the induced spectral density is a~polynomial ensemble on $D$ or $A$).
\end{enumerate}
In part (ii), we also write $f_M = \mathrm{PE}_M(w_1,\hdots,w_n)$ for the density,
where $w_1,\hdots,w_n$ are the one-point weights for $\mathcal{I}_M f_M$.
\end{definition}

For each of our matrix spaces $M \in \{ G,H_2,M_\nu,H_1,H_4 \}$, there exists a subclass 
of polynomial ensembles with nice closure properties under the respective convolution.
We baptize them \emph{P\'olya ensembles} due to their close relation to P\'olya frequency functions 
(as discussed further below).

\pagebreak  

\begin{definition}[P\'olya ensembles]\label{def:PolyaEns}\
\begin{enumerate}[(1)]
\item
A polynomial ensemble on $H_2$ is called a \emph{P\'olya ensemble on~$H_2$} if
\begin{equation}\label{pol-der-1}
w_j(x)=\left(-\frac{\partial}{\partial x}\right)^{j-1}\omega(x)\ \text{ for all }x\in\mathbb{R} \text{ and }j=1,\ldots,n
\end{equation}
with
\begin{multline*}
\omega \in L^{1}_{H_2}(\myreal) :=\biggl\{f\in L^1(\mathbb{R})\biggl| f\text{ is non-negative and $(n-1)$-times differentiable} \\
 \text{and for all $j=0,\ldots,n-1$} : \frac{\partial^jf}{\partial x^j}(x) \in L^1_{[1,n]}(\mathbb{R}) \biggl\}.
\end{multline*}
\item
A polynomial ensemble on $G$ is called a \emph{P\'olya ensemble on $G$} if 
\begin{equation}\label{pol-der-3}
w_j(x)=\left(-x\frac{\partial}{\partial x}\right)^{j-1}\omega(x)\ \text{ for all }x\in\mathbb{R}_+\text{ and }j=1,\ldots,n
\end{equation}
with 
\begin{multline*}
\omega \in L^{1}_{G}(\myreal_+):=\biggl\{f\in L^1(\mathbb{R}_+)\biggl| f\text{ is  non-negative and $(n-1)$-times differentiable}\\
 \text{and for all $j=0,\ldots,n-1$} : \left(x\frac{\partial}{\partial x}\right)^jf(x) \in L^1_{[1,n]}(\myreal_+) \biggl\}.
\end{multline*}
\item	
For $M \in \{ M_\nu,H_1,H_4 \}$,
a polynomial ensemble on $M$ is called a \emph{P\'olya ensemble on $M$} if
\begin{equation}\label{pol-der-2}
w_j(x)=\left(x^{\nu}\frac{\partial}{\partial x}\frac{1}{x^{\nu-1}}\frac{\partial}{\partial x}\right)^{j-1}\omega(x)\ \text{ for all }x\in\mathbb{R}_{+}\text{ and }j=1,\ldots,n
\end{equation}
with
\begin{multline*}
\omega \in L^{1}_{M}(\myreal_+) := \biggl\{f\in L^1(\mathbb{R}_+)\biggl| f\text{ is non-negative and $2(n-1)$-times differentiable,} \\
\text{for all $j=0,\ldots,n-1$} : \left(x^{\nu}\frac{\partial}{\partial x}\frac{1}{x^{\nu-1}}\frac{\partial}{\partial x}\right)^j f(x) \in L^1_{[1,n]}(\mathbb{R}_+) \text{ and } \nonumber\\
\text{for all $l=0,\ldots,n-2$} : \lim_{x\to0} x^{\nu+1}\frac{\partial}{\partial x}\frac{1}{x^\nu}\left(\frac{\partial}{\partial x}x^{\nu+1}\frac{\partial}{\partial x}\frac{1}{x^\nu}\right)^lf(x)=0\biggl\}.
\end{multline*}
\end{enumerate}
In all cases (1) -- (3), we also write $f_M = \mathrm{PE}_M(\omega)$ for the density.
\end{definition}

We want to observe that the differential operator in the first large round brackets in part (3) satisfies $x^\nu\partial_x x^{1-\nu}\partial_x=\partial_x x^{\nu+1}\partial_x x^{-\nu}$. Furthermore, it is equal to $y^{(2\nu-1)/2} \, [\partial_y^2 -(4\nu^2-1)/(4y^2)] \, y^{(1-2\nu)/2}/4$ with $y=\sqrt{x}$, and it simplifies to \linebreak $y^{(2\nu-1)/2}\partial_y^2 y^{(1-2\nu)/2}/4$ for $\nu=\pm1/2$. The latter simplification is related to the~fact that the Bessel functions in the corresponding Hankel transform~\eqref{S2-def} below reduce to the trigonometric functions $\cos(z)/\sqrt{z}$ and $\sin(z)/\sqrt{z}$, respectively.

Note that we use the abbreviation $\mathrm{PE}_M$ both for polynomial ensembles and for P\'olya ensembles. However, this is unlikely to cause confusion, since the definitions coincide for $n = 1$ and the numbers of parameters are different for $n > 1$.

P\'olya ensembles on $G$ and on $H_2$ were already introduced and investigated in \cite{KK:2016a,KK:2016b} and in \cite{KR:2016}, respectively. In those works, they were called \emph{polynomial ensembles of derivative type}. Our motivation to call all these ensembles \emph{P\'olya ensembles} will become clear in Theorem~\ref{thm:rel-Polya} below.

The P\'olya ensembles cover quite a lot of the classical random matrix ensembles
\cite{Mehta}, see also \cite{KK:2016a,KK:2016b,KR:2016}.
We name only a few examples here.

\begin{examples}[``Classical'' P\'olya Ensembles]\label{ex:PolyaEns} \ 
\begin{enumerate}[(a)]
\item (\emph{Gaussian ensembles})
For $M = M_\nu$ or $M = H_\beta$ and $\varepsilon > 0$, 
the matrix density $q_{M,\varepsilon}(y) \propto e^{- \tr(y^* y)/(2\varepsilon)}$ 
defines a P\'olya ensemble on $M$.
The underlying weight function $\omega$ is 
$\omega_\varepsilon(x) \propto x^\nu e^{-x/\varepsilon}$ for $M = M_\nu, H_1, H_4$
and $\omega_\varepsilon(x) \propto e^{-x^2/(2\varepsilon)}$ for $M = H_2$. 
For $M = G$, the log-normal density $\omega_\varepsilon(x) \propto x^{-1} e^{-(\log x)^2/(2\varepsilon)}$
gives rise to a P\'olya ensemble on $M$.
Similarly to the Gaussian distribution in classical probability theory,
these ensembles show up in connection with the~heat kernel 
on the respective matrix spaces, see e.g.\@ Chapter XII.5 in \cite{JL:SLR}.
For this reason, it seems appropriate to call them \emph{Gaussian ensembles}.
\item
The P\'olya ensemble on $H_2$ with $\omega(x)= e^{-(x-\alpha)^2/2}$ is the Gaussian unitary~en\-semble (GUE) 
with shift $\alpha\in\mathbb{R}$, with density $p_{H_2}(y)\propto \exp[-\tr(y-\alpha\eins_n)^2/2]$.
\item	
Consider the classical Laguerre ensemble on $A \simeq \myreal_+^n$ given by
\begin{align}
\label{laguerre}
p_A(a)\propto \det a^\nu \, e^{-\tr a} \, \Theta(a) \, |\Delta_n(a)|^2
\qquad (a \in \myreal^n)\,,
\end{align}
where $\nu > -1$ and $\Theta(y)$ denotes the Heaviside step function for matrices, 
which~is $1$ for $y \in H_2$ positive-definite and $0$ otherwise.
This ensemble induces P\'olya ensembles on~$H_2$ and on~$G$ and,
for $\nu$ an integer, also on~$M_\nu$.

On $H_2$ we get the P\'olya ensemble with the density
$p_{H_2}(y)\propto \det y^\nu e^{-\tr y}\Theta(y)$,
which is also called the \emph{induced Laguerre ensemble}.
The corresponding weight function is
$\omega(x) = x^{n+\nu-1}e^{-x}\Theta(x)$,
with $\Theta(x)$ the ordinary Heaviside step function on $\myreal$.

On the space $G$, we get the P\'olya ensemble with the density
$p_G(g)\propto\det(gg^*)^\nu\exp[-\tr gg^*]$, 
which is also known as the \emph{induced Ginibre ensemble}.
Now the corresponding weight function is $\omega(x) = x^\nu \exp[-x]$.

On $M_\nu$ we get the P\'olya ensemble with the density
$p_{M_\nu}(y)\propto\exp[-\tr yy^*/2]$, 
which is also known as the \emph{chiral Gaussian unitary ensemble}
or \emph{chiral Wishart ensemble}
and which is a special case of the \emph{Gaussian ensemble}
discussed in (a). 
This time, the underlying weight function is also $\omega(x) = x^{\nu} \exp[-x]$.
In contrast to that, for $\mu \ne \nu$, the density \eqref{laguerre}
does not induce a P\'olya en\-semble on $M_\mu$ in~general.

Note that the weight functions in the above P\'olya ensembles are different, 
since different differential operators are applied to~create 
the joint probability density \eqref{laguerre}.

 \item	Further examples of P\'olya ensembles on $G$ are the \emph{induced Jacobi ensemble} ($p_G(g) \propto \det(gg^*)^\nu$ $\det(\eins_n-gg^*)^\mu\Theta(\eins_n-gg^*)$) and the \emph{induced Cauchy-Lorentz ensemble} ($p_G(g) \propto \det(gg^*)^\nu\det(\eins_n+gg^*)^{-2n-\nu-\mu}$) with the weight functions $\omega(x)=x^\nu(1-x)^{n+\mu-1}$ $\Theta(1-x)$ and $\omega(x)=x^\nu/(1+x)^{n+\nu+\mu+1}$, respectively, with $\nu,\mu>-1$.
\end{enumerate}
\end{examples}

\subsection{Convolution Theorems}
The polynomial ensembles on $M$ are special in~that
the relevant multivariate transform is of determinantal form
(see Theorem~\ref{thm:mult-trans-pol}),
and the P\'olya ensembles on $M$ are even more special 
in that the multivariate transform factorizes
(see Corollary \ref{cor:mult-trans-pol}).
The following \emph{convolution theorems} 
are simple consequences of these observations.
Here $w_1,\hdots,w_n$ and $\omega$ are suitable weight functions
as in Definitions \ref{def:PolEns} and \ref{def:PolyaEns},
respectively.

\def\PE{\operatorname{PE}}
\begin{theorem}[Convolution of P\'olya Ensembles and Polynomial Ensembles]\label{thm:pol-conv} \ 
\begin{enumerate}[(1)]
\item See~\cite{KR:2016}.
$\PE_{H_2}(w_1,\hdots,w_n) \ast \PE_{H_2}(\omega) = \PE_{H_2}(w_1 \ast \omega,\hdots,w_n \ast \omega)$.
\item See~\cite{KK:2016b}.
$\PE_{G}(w_1,\hdots,w_n) \circledast \PE_{G}(\omega) = \PE_{G}(w_1 \circledast \omega,\hdots,w_n \circledast \omega)$.
\item
For $M \in \{ M_\nu,H_1,H_4 \}$,  we have
$$\PE_{M}(w_1,\hdots,w_n) \ast \PE_{M}(\omega) = \PE_{M}(w_1 \ast_\nu \omega,\hdots, w_n \ast_\nu \omega).$$
\end{enumerate}
\end{theorem}

\begin{corollary}[Convolution of P\'olya Ensembles]\label{cor:pol-conv}\
\begin{enumerate}[(1)]
\item See~\cite{KR:2016}.
$\PE_{H_2}(\omega) \ast \PE_{H_2}(\chi) = \PE_{H_2}(\omega \ast \chi)$.
\item See~\cite{KK:2016b}.
$\PE_{G}(\omega) \circledast \PE_{G}(\chi) = \PE_{G}(\omega \circledast \chi)$.
\item
For $M \in \{ M_\nu,H_1,H_4 \}$, we have
$\PE_{M}(\omega) \ast \PE_{M}(\chi) = \PE_{M}(\omega \ast_\nu \chi)$.
\end{enumerate}
\end{corollary}

The convolutions on the left hand sides are the matrix convolutions 
introduced in \eqref{addconv} and \eqref{multconv}, while the univariate convolutions on the right sides 
are special cases (for $n = 1$) of the induced convolutions described in \eqref{conv-G} -- \eqref{conv-M}.
The proofs of the Theorem and its Corollary are given in Sec.~\ref{sec:proofs}.

\pagebreak[3]

Let us emphasize that in terms of random matrices, Corollary \ref{cor:pol-conv}
gives the~dis\-tribution of the sum or product of two independent random matrices $X_1$ and $X_2$
on $M$ whose distributions are P\'olya ensembles on $M$.
A similar remark applies to Theorem \ref{thm:pol-conv}.

Finally, as a consequence of Corollary \ref{cor:pol-conv}
and the obvious associativity of all three univariate convolutions, 
each of the three classes of P\'olya ensembles with the corresponding convolution 
becomes a semi-group when adding the appropriate neutral element 
(the Dirac measure in the zero matrix for the additive convolution and the Haar distribution on the unitary group for the multiplicative convolution). Moreover, it has a natural action on the set of all polynomial ensembles on the respective matrix space, as shown by Theorem \ref{thm:pol-conv}.

\subsection{Relation to P\'olya Frequency Functions}
\label{sec:Polya}

We next address the question which weight functions $\omega$
(as in Def.~\ref{def:PolyaEns}) give rise to P\'olya ensembles.
Since this question is trivial for $n = 1$,
we shall always assume that $n \ge 2$.
Our main result shows that for $M = H_2$ and $M = G$,
these weight functions are closely related to P\'olya frequency functions 
\cite{Polya:1913,Polya:1915,Schoen:1951,Karlin:1968},
thereby justifying our name ``P\'olya ensembles''.
For $M \in \{ M_\nu,H_1,H_4 \}$,
the relation is less perfect,
but we~can at~least provide a construction 
by which P\'olya frequency functions
give rise to many (in~fact, infinitely many)
examples of P\'olya ensembles on $M$.

Let us recall the definition of a P\'olya frequency function.

\begin{definition}[P\'olya Frequency Functions, see \cite{Polya:1913,Polya:1915,Schoen:1951}]\label{def:Polya}\
\begin{enumerate}[(a)]
 \item 	Let $N \in \mynat$.
 A measurable function $f$ is called a \emph{P\'olya frequency function of order $N$} if
			\begin{equation}\label{Polya-def}
			\Delta_n(x)\Delta_n(y)\det[f(x_b-y_c)]_{b,c=1,\ldots,n}\geq0
			\end{equation}
			for all $n = 1,\hdots,N$ and all $x,y \in \myreal^n$.
 \item	When Eq.~\eqref{Polya-def} holds for any integer $n\in\mathbb{N}$, $f$ is called a \emph{P\'olya frequency function of infinite order}.
\end{enumerate}
\end{definition}

P\'olya frequency functions (also called \emph{totally positive functions})
 play a role in approximation theory, 
where they give rise to translation-invariant \emph{totally positive kernels}.
The corresponding theory is mainly due to Schoenberg and to Karlin,
based on earlier work by P\'olya.
See Chapter~7 in \cite{Karlin:1968} for background information.

Some examples of P\'olya frequency functions can be found 
in \cite{Karlin:1968} and \cite{Faraut2006}.
Let us highlight a few examples important in random matrix theory.

\begin{examples}[P\'olya Frequency Functions]\
\label{ex:Polya}

\begin{enumerate}[(a)]
	\item	The Gaussian density $f(x)=e^{-(x-\alpha)^2/2}$ with $\alpha\in\mathbb{R}$ is a P\'olya frequency function of infinite order.
	\item	The Heaviside step-function $f(x)=\Theta(x)$ is a P\'olya frequency function of infinite order because all combinations which do not satisfy the interlacing condition $y_1<x_1<y_2<x_2<\ldots<y_n<x_n$ vanish.
	\item	The Heaviside step-function with a monomial $f(x)=x^\nu\Theta(x)$ and $\nu> N-2$ is a P\'olya frequency function of order $N$ because of the identity
				\begin{multline}
				\quad\det[(x_b-y_c)^\nu\Theta(x_b-y_c)]_{b,c=1,\ldots,n}
				=\left(\prod_{j=0}^{n-1}\frac{\Gamma[\nu+1]}{j!\Gamma[\nu-j+1]}\right)\\ \,\times\, \Delta_n(x)\Delta_n(y)\int_{K_2}\det(x-kyk^*)^{\nu-n+1}\Theta(x-kyk^*)d^*k\quad
				\end{multline}
				for all $n\leq N$, see \cite[Theorem 2.3]{KKS:2015}. This function is even a P\'olya frequency function of infinite order for any integer $\nu\geq0$ though for less obvious reasons, cf. Eq.~\eqref{Laplace-Polya.b} below.
	\item The function $f(x)=\exp[-e^{-x}]$ (which is closely related to the density
	of the Gumbel distribution)	is a P\'olya frequency function of infinite order, 
	as can be checked by the Harish-Chandra integral
				\begin{multline}
				\quad\det[\exp[-e^{-(x_b-y_c)}]]_{b,c=1,\ldots,n}=\left(\prod_{j=0}^{n-1}\frac{1}{j!}\right) \\ \,\times\, \frac{\Delta_n(e^x)\Delta_n(e^{y})}{\det e^{(n-1)x}}\int_{K_2}\exp[-\tr ke^{-x}k^*e^{y}]d^*k\,,\quad
				\end{multline}
compare Eq.~\eqref{Harish-Chan} below,
where $e^x$, $e^y$ etc.\@ are defined component-wise and interpreted as diagonal matrices
when necessary.
\end{enumerate}
\end{examples}

We showed the positivity for the latter two examples in a way which already connects the problem to group integrals and random matrix theory. We will return to this idea in Subsection~\ref{sec:groupintegrals}. Alternatively, one can derive many of these examples with the help of the complete characterization of P\'olya frequency functions of infinite order via their Laplace transforms~\cite{Schoen:1951,Karlin:1968}.
As discussed in these references, the Laplace transform of 
an integrable P\'olya frequency function $f$ of infinite order is  of the~form
\begin{multline}\label{Laplace-Polya.a}
\int_{-\infty}^\infty f(x) e^{-sx}dx=C \exp[\gamma s^2-\delta s]\prod_{j=1}^\infty\frac{\exp[\delta_j s]}{1+\delta_j s},\  C>0,\ \gamma\ge0,\\
 \delta,\delta_j\in\mathbb{R},\ 0<\gamma+\sum_{j=1}^\infty \delta_j^2<\infty,\ {\rm and}\ \min_{\delta_j>0}\left\{\frac{1}{\delta_j}\right\}>-{\rm Re}\, s>\max_{\delta_j<0}\left\{\frac{1}{\delta_j}\right\},
\end{multline}
when the support of $f$ is contained in $\myreal$
and of the form
\begin{multline}\label{Laplace-Polya.b}
\int_{0}^\infty f(x) e^{-sx}dx=C \exp[-\delta s] \prod_{j=1}^\infty\frac{1}{1+\delta_j s},\ C > 0, \\
\delta,\delta_j\geq0,\, 0<\sum_{j=1}^\infty \delta_j<\infty,\ {\rm and}\ \min_{\delta_j>0}\left\{\frac{1}{\delta_j}\right\}>-{\rm Re}\, s,
\end{multline}
when the support of $f$ is contained in $[0,\infty[$. In particular, the reciprocal Laplace transform of $f$
is an entire function in either case. Unfortunately, there is no such explicit characterization for P\'olya frequency functions of a finite order, although they may yield some interesting random matrix ensembles.
For instance, similarly as in Example \ref{ex:Polya}\,(c), the function $f(x) = x^\nu e^{-x} \, \Theta(x)$ is a P\'olya frequency function of order $N$ for any $\nu > N-2$.
The Laplace transform of this example is given by $\int_0^\infty x^\nu e^{-(1+s)x} dx = \Gamma(\nu+1)/(1+s)^{\nu+1}$,
which is \emph{not} of the form \eqref{Laplace-Polya.a} or \eqref{Laplace-Polya.b} when $\nu$ is not an integer.
Since integrable P\'olya frequency functions of a fixed finite order $N$ are closed under additive convolution~\cite{Schoen:1951,Karlin:1968}, we can at least say that if $f$ is an integrable function such that
\begin{multline}\label{Laplace-Polya.c}
\int_{-\infty}^\infty f(x) e^{-sx}dx=C \exp[\gamma s^2-\delta s]\prod_{j=1}^\infty\frac{\exp[\nu_j\delta_j s]}{(1+\delta_j s)^{\nu_j}},\quad C>0,\ \nu_j,\gamma\geq0,\\
 \delta,\delta_j\in\mathbb{R},\ 0<\gamma+\sum_{j=1}^\infty \nu_j\delta_j^2<\infty,\ {\rm and}\ \min_{\delta_j>0}\left\{\frac{1}{\delta_j}\right\}>-{\rm Re}\, s>\max_{\delta_j<0}\left\{\frac{1}{\delta_j}\right\},
\end{multline}
or
\begin{multline}\label{Laplace-Polya.d}
\int_{0}^\infty f(x) e^{-sx}dx=C \exp[-\delta s]\prod_{j=1}^\infty\frac{1}{(1+\delta_j s)^{\nu_j}},\quad C>0,\ \nu_j \ge 0,\\
 \delta,\delta_j\geq 0,\ 0<\sum_{j=1}^\infty \nu_j\delta_j<\infty,\ {\rm and}\ \min_{\delta_j>0}\left\{\frac{1}{\delta_j}\right\}>-{\rm Re}\, s
\end{multline}
then $f$ is a P\'olya frequency function of order $N$ with $N-1$ smaller than all non-integer exponents $\nu_j$. The problem is that Eqs.~\eqref{Laplace-Polya.c} and \eqref{Laplace-Polya.d} are not exhaustive. For example, the P\'olya frequency functions of order $N=1$ are the positive functions, and the P\'olya frequency functions of order $N=2$ are the positive and log-concave functions, see \cite{Karlin:1968}.

For fixed $n \ge 2$, we will prove Theorems \ref{thm:rel-Polya} and \ref{thm:rel-Polya.b}
below in Section~\ref{sec:thmproof}. Before we state our results, let us recall the main problem. Any function $\omega$ with suitable differentiability and integrability properties as in Def.~\ref{def:PolyaEns} defines a ``signed'' P\'olya ensemble on $M$, meaning that the associated density as in Eq.~\eqref{jpdf-pol-ens}
(but without the prefactor) is not necessarily non-negative.
The hard question is: What are the necessary or sufficient conditions on $\omega$ such that this density is also non-negative, and hence gives rise to a random matrix ensemble after proper normalization? The~following theorem answers this question completely for P\'olya ensembles on $H_2$ and on $G$. 

\begin{theorem}[Relation to P\'olya Frequency Functions]\label{thm:rel-Polya}\
\begin{enumerate}
	\item	Let $\omega\in L^{1}_{H_2}(\mathbb{R})$ with $\omega\neq0$. Then, $\omega$ gives rise to a P\'olya ensemble on $H_2$ if and only if it is a P\'olya frequency function of order $n$.
	\item	Let $\omega\in L^{1}_{G}(\mathbb{R}_+)$ with $\omega\neq0$. Then, $\omega$ gives rise to a P\'olya ensemble on $G$ if and only if $\tilde\omega(x) := \omega(e^{-x})e^{-x}$ is a P\'olya frequency function of order $n$.
\end{enumerate}
\end{theorem}

Thus, P\'olya ensembles on $H_2$ or on $G$ are closely related to P\'olya frequency functions for which the associated function $\omega$ is an element of $L_{H_2}^1(\myreal)$ or $L_{G}^1(\myreal_+)$, respectively. In fact, in view of Cor.~\ref{cor:mult-trans-pol} below, the function $\omega$ is determined by the corresponding P\'olya ensemble on $H_2$ or on $G$ up to a scalar factor. Thus, the relations just described become bijections if we restrict ourselves to \emph{normalized} P\'olya frequency functions (i.e., P\'olya frequency functions of Lebesgue integral $1$) from the respective classes.

The first statement in Theorem \ref{thm:rel-Polya} is closely related to several similar results in \cite{Karlin:1957,Karlin:1968}, although we have not been able to find a result entailing the precise formulation given above. Moreover, the first statement is well-known in the situation where $\omega$ is assumed to define a P\'olya ensemble on $H_2$ for \emph{any} $n \in \mynat$, see \cite{Pickrell:1991,OV:1996,Faraut2006}. The relation claimed in the second statement was implicitly used in a previous work of ours~\cite{KK:2016b} to show that the analytic continuation of matrix product ensembles in the number of Ginibre factors does not always yield a probability density.

For $M \in \{ H_1,H_4,M_\nu \}$, we do not have a complete characterization,
but we~have at least the following partial result.

\begin{theorem}[Sufficient condition for P\'olya ensembles on $M \in \{ H_1,H_4,M_\nu \}$]\label{thm:rel-Polya.b}
Let $M \in \{ H_1,H_4,M_\nu \}$, and let $\tilde\omega\in L^{1}_{H_2}(\mathbb{R})$ be a P\'olya frequency function of order $n$ with $\tilde\omega \ne 0$ and support contained in $[0,\infty[$. Then the function
				\begin{equation}\label{polya-rel.1}
				\omega(x)=\frac{1}{\Gamma[\nu+1]}\int_0^\infty \left(\frac{x}{y}\right)^\nu \exp\left[-\frac{x}{y}\right]\tilde\omega(y)\frac{dy}{y}\in L^{1}_{M}(\mathbb{R}_+)
				\end{equation}
				gives rise to a P\'olya ensemble on $M$.
\end{theorem}

This theorem shows that even P\'olya ensembles on $M$ are closely related to P\'olya frequency functions
with support contained in $\mathbb{R}_+$, in the sense
that the latter give rise to a large number of examples. In particular Eqs.~\eqref{Laplace-Polya.b} and \eqref{Laplace-Polya.d} in combination with Eq.~\eqref{polya-rel.1} yield many ensembles of this kind.
Theorem \ref{thm:rel-Polya.b} can be found by noticing that the Hankel transform of $\omega$ (as in Section \ref{sec:transforms} below) is equal to the Laplace transform of $\tilde{\omega}$. Moreover, Theorem \ref{thm:rel-Polya.b} should be compared with \cite[Corollary~5.2]{Kallenberg:2012}, which provides a related characterization 
for certain \emph{infinite-dimensional} random matrices.

Unfortunately, the construction in Theorem \ref{thm:rel-Polya.b} does not produce all P\'olya en\-sembles on $M$. For example, to recover the Gaussian ensemble on $M$ (see Ex~\ref{ex:PolyaEns}\,(a)), which corresponds to the weight function $\omega(x)=x^\nu e^{-x}$, we would have to choose $\tilde\omega(y)=\Gamma[\nu+1]\delta(y-1)$, which is a distribution and not a function. 
Even worse, the weights $\omega$ arising in Theorem \ref{thm:rel-Polya.b} 
satisfy $\omega(x) > 0$ for all $x > 0$, 
but there also exist examples of P\'olya ensembles on $M$ without this~property:

\begin{example}[Example of a P\'olya ensemble on $M_0$ beyond Theorem \ref{thm:rel-Polya.b}]
Set~$\omega(x) := e^{-\frac{1}{a-x}} \, \pmb{1}_{(0,a)}(x)$,
where $a \in (0,\tfrac14)$.
Then it is straightforward to show that $\omega$
satisfies the conditions of Def.\@ \ref{def:PolyaEns}\,(3)
(including $p_{M_0}(\omega) \ge 0$) with~$n = 2$ and $\nu = 0$, 
and hence defines a P\'olya ensemble on $M_0$.
\end{example}

\pagebreak[2]

We conclude this subsection with some applications of Theorems \ref{thm:rel-Polya} and \ref{thm:rel-Polya.b}:

\begin{examples}[Non-Trivial P\'olya Ensembles]\

\begin{enumerate}[(a)]\label{ex:1}
	\item	
Since the deformed Gumbel density $\omega(x)=\exp[-e^{-x}-\alpha x]$ with $\alpha>0$ 
is a P\'olya frequency function of infinite order \cite{Faraut2006}, it defines a P\'olya ensemble on $H_2$. Moreover, applying Thm.~\ref{thm:rel-Polya}\,(2) to this density, we recover the P\'olya ensemble on $G$ induced by some Laguerre ensemble, compare Ex.~\ref{ex:PolyaEns} (c).
	\item	The combination of the Gaussian density as a P\'olya frequency function (P\'olya ensemble on $H_2$) and Thm.~\ref{thm:rel-Polya}\,(2) yields the log-normal distribution $\omega(x)= x^{-1} \exp[-({\rm ln}\,x-\alpha)^2/(2\sigma^2)]$. The associated P\'olya ensemble on $G$ was already identified as a particular form of a Muttalib-Borodin ensemble~\cite{Muttalib,Borodin,Forrester-Wang} in \cite{KK:2016a,KK:2016b}.
	\item	The function $\omega(x)=\cosh^{-\mu}(x)$ is a P\'olya frequency function
    of infinite order for any real number $\mu > 0$, see~also \cite{Faraut2006} for $\mu=1,2$, 
    and hence gives rise to a P\'olya ensemble on $H_2$. The joint probability density of the eigenvalues 
    associated with $\omega(x)$ takes the form
	\begin{equation}
	\Delta_n(x)\det[(-\partial_b)^{a-1}\omega(x_b)]\propto\Delta_n(x)\Delta_n(e^{2x})
	\prod_{j=1}^{n} \frac{\exp[-(n-1)x_j]}{\cosh^{\mu+n-1}(x_j)}.
	\end{equation}
	Thence, it yields a particular Muttalib-Borodin ensemble~\cite{Muttalib,Borodin,Forrester-Wang}. With the aid of Thm~\ref{thm:rel-Polya}\,(2), one can show that
	for $\mu > n-1$, it corresponds to a particular form of a Cauchy-Lorentz ensemble as a P\'olya ensemble on $G$.
	\item
	The P\'olya ensemble on $H_2$ corresponding to the induced Jacobi ensemble 
	(a P\'olya ensemble on $G$, see Example \ref{ex:PolyaEns}(d)) is associated to the P\'olya frequency function $\omega(x)=\exp[-\alpha x]\sinh^\mu(x)\Theta(x)$ with $\alpha>\mu>n-2$. This function also yields a Muttalib-Borodin ensemble with an exponential function in the second Vandermonde determinant. Moreover, in combination with Eq.~\eqref{polya-rel.1}, we obtain a non-trivial P\'olya ensemble on $M_\nu$ associated with the function
	\begin{equation}
	\hat{\omega}(x)=\int_0^\infty \left(\frac{x}{y}\right)^\nu\exp\left[-\frac{x}{y}-\alpha y\right]\sinh^\mu(y)\frac{dy}{y} \,.
	\end{equation}
	
\pagebreak[2]

	\item	Using Eq.~\eqref{polya-rel.1} with the weight function $\tilde{\omega}$ for the induced Laguerre ensemble on $H_2$, we find that $\omega(x)=x^{(\mu+\nu)/2}K_{\mu-\nu}(2\sqrt{x})$ with $\mu > n - 2$ and $K_j$ the~modified Bessel function of the second kind gives rise to a P\'olya ensemble on $M_\nu$.
\end{enumerate}
\end{examples}

\smallskip

\subsection{Identities Involving Group Integrals} 
\label{sec:groupintegrals}

Another nice feature of P\'olya en\-sembles  is that they satisfy 
interesting identities involving group integrals.

\begin{theorem}[Group Integrals and P\'olya Ensembles]\label{thm:group-pol} \ 
\begin{enumerate}[(1)]
 \item Let $p_{H_2} = \PE_{H_2}(\omega)$. Then
 			\begin{equation}\label{group-int4}
 			\Delta_n(y)\Delta_n(x)\int_{K_2}p_{H_2}(y-kxk^*)d^*k=\frac{1}{n! \, C_{H_2}} \frac{\det[\omega(y_b-x_c)]_{b,c=1,\ldots,n}}{(\fourier \omega(0))^n}
 			\end{equation}
 for almost all $x,y \in D$, with $C_{H_2}$ as in Eq.~\eqref{const}
 and $\fourier \omega$ the Fourier transform as in Eq.~\eqref{F-def} below.
 
 \item Let $p_G = \PE_G(\omega)$. Then
 			\begin{equation}\label{group-int3}
 			\Delta_n(y)\Delta_n(-x^{-1})\int_{{\rm U}(n)}p_G(x^{-1/2}ky^{1/2})d^*k=\frac{1}{n! \, C_G} \frac{\det[\omega(y_bx_c^{-1})]_{b,c=1,\ldots,n}}{\prod_{j=1}^{n} \mellin \omega(j)}
 			\end{equation}
 			for almost all $x,y \in A$, with $x^{1/2},y^{1/2}$ defined componentwise,  $C_G$ as in~Eq.~\eqref{const}
 and $\mellin \omega$ the Mellin transform as in Eq.~\eqref{M-def} below.

 \item For $M \in \{ M_\nu,H_1,H_4 \}$, let $p_M = \PE_M(\omega)$. Then
 			\begin{multline}\label{group-int1}
 			\Delta_n(y)\Delta_n(x)\int_{K}p_{M}(\iota_M(y)-k\iota_M(x)k^*)d^*k\\
 			=\frac{C_{M(1)}^n}{n! \, C_M} \det\bigg[\int_{K(1)}p_{M(1)}(\iota_{M(1)}(y_b)-k\iota_{M(1)}(x_c)k^*)d^*k\bigg]_{b,c=1,\ldots,n}
 			\end{multline}
 			for almost all $y,x\in A$, where $C_M$ is as in Eq.~\eqref{const},
 			$M(1)$, $K(1)$ and $C_{M(1)}$ are the matrix space, the matrix group
 			and the constant for $n=1$, and $p_{M(1)} = \PE_{M(1)}(\omega)$ 
 			$(${with the same weight function $\omega$}$)$.
\end{enumerate}
\end{theorem}

\emph{Remark.} By the invariance of the matrix densities, 
the above identities continue to hold if we replace the diagonal matrices $x$ and $y$
inside the group integrals by general matrices $x$ and $y$ in $H_2$, $G$ or $M$, respectively,
with the same eigenvalues or squared singular values.

Theorem \ref{thm:group-pol} will be proved in Section \ref{sec:groupintegrals-proof},
and also plays an important role in the proofs of Theorems~\ref{thm:rel-Polya} and~\ref{thm:rel-Polya.b}.

The identities are far from trivial in quite a few cases, see the following examples.

\begin{examples}[Some Group Integrals]\label{ex:2}\
For a differentiable function $F:B\to\mathbb{R}$ with first derivative $F'(x)\neq0$ for all $x\in B\subset\mathbb{R}$, we define the abbreviation
\begin{multline}
\widetilde{F}(y) := 
\sqrt{\frac{1}{\det F'(y)} \det \frac{F(y) \otimes \eins_n - \eins_n \otimes F(y)}{y \otimes \eins_n - \eins_n \otimes y}} \\
:=
\sqrt{\prod_{j < k} \frac{(F(a_k) - F(a_j))^2}{(a_k - a_j)^2}}
=
\frac{\Delta_n(F(a))}{\Delta_n(a)}
\end{multline}
for all $y\in H_2$ with eigenvalues $a_1,\hdots,a_n \in B$.

\begin{enumerate}[(a)]
	\item	The P\'olya ensemble on $G$ defined by the weight $\omega(x)= \exp[-({\rm ln}\,x-\alpha)^2/(2\sigma^2)]$ (a Muttalib-Borodin ensemble) has the matrix density
				\begin{equation}
				p_G(g)\propto \widetilde{\rm ln}(gg^*)\det (gg^*)^{\alpha/\sigma^2}\exp[-\tr({\rm ln}\,gg^*)^2/(2\sigma^2)]
				\end{equation}
				and satisfies the identity
				\begin{multline}
				\Delta_n(y)\Delta_n(x^{-1})\int_{{\rm U}(n)}p_G(x^{-1/2}ky^{1/2})d^*k \\ \propto\det\left[\exp\left[-({\rm ln}\,y_b-{\rm ln}\,x_c-\alpha)^2)/(2\sigma^2)\right]\right]_{b,c=1,\ldots,n}.
				\end{multline}
	\item  The P\'olya ensemble on $H_2$ defined by the weight $\omega(x)= \cosh^{-\mu}(x)$ 
	(another Muttalib-Borodin ensemble) has the matrix density
				\begin{equation}
				p_{H_2}(y)\propto \widetilde{\rm exp}(2y)\frac{\exp[-(n-1)\tr y]}{\det[\cosh^{\mu+n-1}y]}
				\end{equation}
				(where the matrix in the determinant is defined by spectral calculus)
				and satisfies the identity
 			\begin{equation}
 			\Delta_n(y)\Delta_n(x)\int_{K_2}p_{H_2}(y-kxk^*)d^*k\propto\det\left[\cosh^{-\mu}(y_b-x_c)\right]_{b,c=1,\ldots,n}.
 			\end{equation}
\end{enumerate}
\end{examples}

\smallskip


\section{Multivariate Transforms}
\label{sec:transforms}

To motivate our approach, let us begin with an exposition of the Hankel transform
tailored to our needs. Fix $\nu \in \mynat_0$. Then, for $f \in L^1(\mycmplx^{1+\nu})$, 
its \emph{Fourier~trans\-form} is~defined~by 
\begin{align}
\label{fourier0}
(\fourier f)(y) = \int_{\mycmplx^{1+\nu}} f(x) e^{\imath(x^*y + y^*x)} \, dx \,, \qquad y \in \mathbb{C}^{1+\nu} \,.
\end{align}
Now suppose that $f$ is additionally unitarily invariant, 
i.e.\@ $f(Ux) = f(x)$ for any $U \in \mathrm{U}(1+\nu)$
or (equivalently) $f(x) = F(\|x\|)$ with $\|x\|$ the Euclidean norm.
Then $\fourier f$ is also unitarily invariant,
and it is natural to express everything in terms of 
the squared radii $r := \|x\|_2^2$ and $s := \|y\|_2^2$.
To this end, set
\begin{align}
\label{radialdensity}
\tilde{f}(r) := \frac{\pi^{\nu+1}}{\Gamma(\nu+1)} f(\sqrt{r} e_1) r^{\nu} \qquad (r > 0) \,,
\end{align}
which is the induced density of the squared radius $r := \|x\|^2$
in case $f$ is a prob\-ability density.
By unitary invariance, it suffices to compute $(\fourier f)(y)$
for $y = \sqrt{s} e_1$, with $e_1 := (1,0,\hdots,0) \in \mycmplx^{1+\nu}$.
After a moderate calculation, one finds that
\begin{align}
\label{hankel0}
(\fourier f)(\sqrt{s} e_1) = \Gamma(\nu+1) \int_0^\infty \tilde{f}(r) \frac{J_\nu(2\sqrt{rs})}{(rs)^{\nu/2}} \, dr =: (\hankel_\nu \tilde{f})(s) \,, 
\end{align}
where $J_\nu$ is the Bessel function of parameter $\nu$ and $\hankel_\nu$ is (a version of) the Hankel transform.
Let us note that the Hankel transform is usually defined a bit differently in the literature,
so that it becomes an involution.
Also, let us note that if $f,g \in L^1(\mycmplx^{1+\nu})$ are unitarily invariant,
then $f \ast g \in L^1(\mycmplx^{1+\nu})$ is also unitarily invariant,
and the well-known relation
$
\fourier (f \ast g) = \fourier f \cdot \fourier g
$
for the Fourier transform translates into the relation
$
\hankel_\nu (\widetilde{f \ast g}) = \hankel_\nu \tilde{f} \cdot \hankel_\nu \tilde{g}
$
for the Hankel transform.

Indeed, in view of the identification $\mycmplx^{1+\nu} \simeq {\rm Mat}_{\mycmplx}(1,1+\nu) \simeq M_\nu(1) =: M$,
the~radial density $\tilde{f}$ is nothing else than the spectral density $\mathcal{I}_{M}(f)$
introduced in Sub\-section \ref{sec:matrixdensities}, and the relation for the Hankel transform 
just derived takes the form
\begin{align}
\label{hankel1}
\hankel_\nu (f_A \ast_\nu g_A) = (\hankel_\nu f_A) \cdot (\hankel_\nu g_A) \,,
\end{align}
with $f_A$, $g_A$ and $\ast_\nu$ as in Eq.~\eqref{conv-M}. By similar reasoning, Eq.~\eqref{hankel1}
also holds for the~matrix spaces $H_1$ and $H_4$, again with $n = 1$ and the parameter $\nu$
as in \eqref{I-M}.

We will now develop a similar approach for the matrix spaces $M_\nu$ listed in Sub\-section \ref{sec:matrixspaces},
where it seems natural to express $K$-invariant functions in terms of their eigenvalues 
or squared singular values, and hence in terms of the induced spectral densities
as introduced in Subsection \ref{sec:matrixdensities}.

For the Hermitian matrix spaces $M \in \{ H_1,H_2,H_4,M_\nu \}$ and $f_M \in L^1(M)$, 
the~\emph{Fourier transform} is given by
\begin{align}
\label{fourier1}
(\fourier f_M)(x) := \int_M f_M(y) \exp(\imath \tr xy) \, dy \,, \qquad x \in M  \,.
\end{align}
When $f_M$ is additionally $K$-invariant, then $\fourier f_M$ is also $K$-invariant.
Suppose that $s=\text{diag}(s_1,\ldots,s_n)$ $\in\mathbb{C}^{n \times n}$ with $s_i\neq s_j$ for $i \ne j$
such that the respective integrals exist in the Lebesgue sense.
Then the Fourier transform simplifies to 
\begin{align}
(\fourier f_{H_2})(s) &= \int_{H_2} f_{H_2}(y)\left(\int_{K_2}  \exp[\imath\tr yksk^*]d^*k\right)dy \nonumber\\
&= \prod_{j=0}^{n-1}j!\int_{D} f_{D}(a) \frac{\det\big[\exp[\imath a_bs_c]\big]_{b,c=1,\ldots,n}}{\Delta_n(\imath s)\Delta_n(a)} da =: (\fourier f_D)(s)
\label{S1-def}
\end{align}
with $f_D := \mathcal{I}_{H_2}(f_{H_2}) \in L^{1,\mathbb{S}}(D)$ for $M = H_2$
and to 
\begin{align}
(\fourier f_M)(\iota_M(s)) &= \int_{M}f_{M}(y)\left(\int_{K}  \exp\left[\imath\tr k^*yk\iota_M(s)\right]d^*k\right)dy \nonumber\\
&= \prod_{j=0}^{n-1}\Gamma[j+\nu+1]j!\int_{A} f_A(a) \frac{\det\big[J_{\nu}(2\sqrt{a_bs_c})/\left(a_bs_c\right)^{\nu/2}\big]_{b,c=1,\ldots,n}}{\Delta_n(- s)\Delta_n(a)}da \nonumber\\
&=: (\hankel_\nu f_A)(s) 
\label{S2-def}
\end{align}
with $f_A := \mathcal{I}_{M}(f_M) \in L^{1,\mathbb{S}}(A)$ for $M \in \{ H_1,H_4,M_\nu \}$.
Here, $\nu$ and $\iota_M(s)$ are defined similarly as in \eqref{I-M}.
Furthermore, for the space $G$ and $f_G \in L^{1,K}(G)$, we~may consider the \emph{spherical transform} defined by
\begin{align}
(\mathcal{S} f_G)(s)&=\int_{G}f_G(g)\bigg(\int_{{\rm{U}}(n)}  \prod_{j=1}^n \det (\Pi_{j,n} kgg^*k^* \Pi_{j,n}^*)^{s_j-s_{j+1}-1} d^*k\bigg)dg \nonumber\\
 			&=\prod_{j=0}^{n-1}j!\int_{A} f_A(a) \frac{\det\big[ a_b^{s_c-(n+1)/2} \big]_{b,c=1,\ldots,n}}{\Delta_n(s) \Delta_n(a)} \, da =: (\mellin f_A)(s),
\label{S3-def}
\end{align}
where $s_{n+1} := \frac{n-1}{2}$ and $f_A := \mathcal{I}_{G}(f_G) \in L^{1,\mathbb{S}}(A)$.

To obtain the simplifications in Eqs.~\eqref{S1-def}, \eqref{S2-def}, and \eqref{S3-def},
we have used the Harish-Chandra-Itzykson-Zuber integral~\cite{HC,IZ} for $K_2=\U(n)$,
\begin{equation}\label{Harish-Chan}
\int_{K_2} \exp[\imath\tr aksk^*]d^*k=\prod_{j=0}^{n-1}j!\frac{\det[\exp[\imath a_bs_c]]_{b,c=1,\ldots,n}}{\Delta_n(\imath s)\Delta_n(a)},
\end{equation}
the Harish-Chandra integral \cite{HC} for $K=K_1,K_4$ 
and the Berezin-Karpelevich integral~\cite{BK,GW} for $K=\hat{K}_\nu$,
\begin{equation}\label{Berezin-Karpelevich}
\int_{K} \exp\left[\imath\tr k^*ak\iota_{M}(s)\right]d^*k=\prod_{j=0}^{n-1}\Gamma[j+\nu+1]j!\frac{\det\left[J_{\nu}(2\sqrt{a_bs_c})/\left(a_bs_c\right)^{\nu/2}\right]_{b,c=1,\ldots,n}}{\Delta_n(- s)\Delta_n(a)},
\end{equation}
 and the Gelfand-Na\u{\i}mark integral~\cite{GelNai} for ${\rm U}(n)$,
\begin{equation}\label{Gelfand-Naimark}
\int_{{\rm U}(n)}  \prod_{j=1}^n \det (\Pi_{j,n} kak^* \Pi_{j,n}^*)^{s_j-s_{j+1}-1} d^*k=\prod_{j=0}^{n-1}j! \frac{\det[a_b^{s_c-(n+1)/2}]_{b,c=1,\ldots,n}}{\Delta_n(s) \Delta_n(a)},
\end{equation}
with $s_{n+1} := \frac{n-1}{2}$ as above.
Recall that we work with the induced spectral densities, so that the Jacobians resulting from the diagonalization of the matrices are absorbed in $f_D$ and $f_A$, respectively.

In view of the preceding results, we make the following definition:

\begin{definition}[Multivariate Transforms]\label{def:multi-trans}
The induced transforms in \eqref{S1-def}, \eqref{S2-def} and \eqref{S3-def}
are called the \emph{multivariate Fourier transform}, the \emph{multivariate Hankel transform}
and the \emph{multivariate Mellin transform}, respectively.
\end{definition}

The multivariate transforms are normalized in such a way that
\begin{equation}\label{normalization-def}
\fourier f_D(0)=\int_{H_2}f_{H_2}(y)dy,\ \hankel_{\nu}f_A(0)=\int_{M}f_M(y)dy,\ \mellin f_A(s^{(0)})=\int_{G}f_G(g)dg,
\end{equation}
where the functions are as in \eqref{S1-def}, \eqref{S2-def} and \eqref{S3-def}
and $s_j^{(0)} := (j-1) + (n+1)/2$, \linebreak $j=1,\hdots,n$.
Our motivation for calling the multivariate transforms $\fourier$, $\mellin$ and~$\hankel_\nu$ is that,
for~$n=1$, they reduce to the univariate Fourier, Hankel and Mellin trans\-form defined by 
\begin{align}
\label{F-def}
\fourier f(s)=&\int_{-\infty}^\infty f(x)\exp[\imath xs]dx, 
&& f \in L^1(\myreal),
\\ 			
\label{H-def}
\hankel_\nu f(s)=&\Gamma[\nu+1]\int_{0}^\infty f(x) \frac{J_\nu(2\sqrt{xs})}{(sx)^{\nu/2}} dx,
&& f\in L^1(\mathbb{R}_+)\,,
\\
\label{M-def}
\mellin f(s)=&\int_{0}^\infty f(x) x^{s-1}dx,
&& f\in L^1(\mathbb{R}_+),
\end{align}
respectively. Note that Eq.~\eqref{H-def} coincides with Eq.~\eqref{hankel0}.

Also note that, by Eqs. \eqref{S1-def} and \eqref{S2-def}, the induced transforms $\fourier$ and $\hankel_{\nu}$ are essentially the ordinary Fourier transform \eqref{fourier1} restricted to the matrices $s$ and $\iota_M(s)$ with $s$ diagonal, which is sufficient by $K$-invariance. Thus, in particular, they~inherit the injectivity of the ordinary Fourier transform. The transform $\mellin$ in \eqref{S3-def} is also injective on $L^{1,\mathbb{S}}(A)$, but for less simple reasons, see \eg \cite{Helgason3,KK:2016a}. Finally, let us note that all the transforms in Def.~\ref{def:multi-trans} may be inverted fairly explicitly (possibly after appropriate regularization) either as a consequence of the Fourier inversion formula or by the inversion formula for the spherical transform \cite{Helgason3,KK:2016a}; see also the proof of Theorem \ref{thm:group-pol} further below.

The multivariate transforms in Def.~\ref{def:multi-trans} have the important property
that they convert the induced convolutions in Eqs.\@ \eqref{conv-G} -- \eqref{conv-M} 
into multiplications. For the Fourier transform and the Hankel transform, this follows
from the corresponding property of the Fourier transform \eqref{fourier1}.
For the spherical transform, see \eg \cite[Proof of Lemma IV.3.2]{Helgason3}. 

\begin{theorem}[Multiplication Theorems for the Convolutions]\label{thm:conv-fact}\ 

\begin{enumerate}[(1)]
 \item	For $f_D,h_D \in L^{1,\mathbb{S}}(D)$, $\mathcal{F}[f_D \ast h_D]=\mathcal{F}f_D\,\mathcal{F}h_D$.
 \item	For $f_A,h_A \in L^{1,\mathbb{S}}(A)$, $\mathcal{M}[f_A \circledast h_A]=\mathcal{M}f_{A}\,\mathcal{M}h_{A}$.
 \item	For $f_A,h_A \in L^{1,\mathbb{S}}(A)$, $\mathcal{H}_{\nu}[f_A \ast_{\nu} h_A]=\mathcal{H}_{\nu}f_{A}\,\mathcal{H}_{\nu}h_{A}$.
\end{enumerate}
\end{theorem}

Similarly as the univariate transforms, these multiplication theorems come in handy for studying the three matrix convolutions, and often allow for explicit results in special examples. In particular, they are crucial for identifying the subsets of polynomial ensembles which are closed under these convolutions, see Subsection~\ref{sec:PolEns}.

The reason why polynomial ensemble and P\'olya ensembles are so special becomes clear when taking the multivariate transforms~\ref{def:multi-trans}.

\begin{theorem}[Multivariate Transforms of Polynomial Ensembles]\label{thm:mult-trans-pol}\
 \begin{enumerate}[(1)]
 \item See~\cite{KR:2016}.	For $p_{H_2} = \PE_{H_2}(w_1,\hdots,w_n)$
and $p_D^{} := \mathcal{I}_{H_2} (p_{H_2})$,
 			\begin{equation}\label{S1-pol}
 			\mathcal{F}p_D^{}(s)=C_n[w]\left(\prod_{j=1}^{n}j!\right)\frac{\det[\mathcal{F}w_b(s_c)]_{b,c=1,\ldots,n}}{\Delta_n(\imath s)}.
 			\end{equation}
 \item See~\cite{KK:2016b}.	For $p_{G} = \PE_{G}(w_1,\hdots,w_n)$
 and $p_A^{} := \mathcal{I}_{G} (p_{G})$,
 			\begin{equation}\label{S3-pol}
 			\mathcal{M} p_A^{}(s)=C_n[w]\left(\prod_{j=1}^{n}j!\right)\frac{\det[\mathcal{M}w_b(s_c-(n-1)/2)]_{b,c=1,\ldots,n}}{\Delta_n(s)}.
 			\end{equation}
 \item	For $M \in \{ M_\nu,H_1,H_4 \}$, $p_{M} = \PE_{M}(w_1,\hdots,w_n)$
and $p_A^{} := \mathcal{I}_{M} (p_{M})$,
 			\begin{equation}\label{S2-pol}
 			\mathcal{H}_\nu p_{A}^{}(s)=C_n[w]\left(\prod_{j=1}^{n}\frac{j!\Gamma[j+\nu]}{\Gamma[1+\nu]}\right)\frac{\det[\mathcal{H}_\nu w_b(s_c)]_{b,c=1,\ldots,n}}{\Delta_n(-s)}.
 			\end{equation}
 \end{enumerate}
\end{theorem}
\begin{proof}
 For all three statements we have to plug Eq.~\eqref{jpdf-pol-ens} into the second lines of Eqs.~\eqref{S1-def}, \eqref{S2-def} and \eqref{S3-def}. The integrals can be performed by applying Andr\'eief's integration theorem~\cite{Andreief} and by identifying the entries in the remaining determinant with the univariate transforms in Eqs.~\eqref{F-def}, \eqref{H-def} and \eqref{M-def}, which concludes the proof.
\end{proof}

\begin{corollary}[Multivariate Transforms of P\'olya Ensembles]\label{cor:mult-trans-pol}\
\begin{enumerate}[(1)]
 \item	    See~\cite{KR:2016}.
 For $p_{H_2} = \PE_{H_2}(\omega)$
and $p_D^{} := \mathcal{I}_{H_2} (p_{H_2})$,
 			\begin{equation}\label{S1-pol-der}
 			\mathcal{F} p_D^{}(s)=\prod_{j=1}^n\frac{\mathcal{F}\omega(s_j)}{\mathcal{F}\omega(0)} \quad \text{and} \quad C_n[\omega]=\prod_{j=1}^{n}\frac{1}{j!\mathcal{F}\omega(0)}.
 			\end{equation}
 \item	    See~\cite{KK:2016b}.
 For $p_{G} = \PE_{G}(\omega)$
and $p_A^{} := \mathcal{I}_{G} (p_{G})$,
			\begin{equation}\label{S3-pol-der}
 			\mathcal{M} p_A^{}(s)=\prod_{j=1}^n\frac{\mathcal{M}\omega(s_j-(n-1)/2)}{\mathcal{M}\omega(j)} \quad \text{and} \quad C_n[\omega]=\prod_{j=1}^{n}\frac{1}{j!\mathcal{M}\omega(j)}.
 			\end{equation}
 \item	
 For $M \in \{ M_\nu, H_1, H_4 \}$, $p_M = \PE_M(\omega)$ 
and $p_A^{} := \mathcal{I}_{M} (p_{M})$,
 			\begin{equation}\label{S2-pol-der}
 			\mathcal{H}_{\nu} p_A^{}(s)=\prod_{j=1}^n\frac{\mathcal{H}_\nu\omega(s_j)}{\mathcal{H}_\nu\omega(0)} \quad \text{and} \quad C_n[\omega]=\prod_{j=1}^{n}\frac{\Gamma[1+\nu]}{j!\Gamma[j+\nu]\mathcal{H}_\nu\omega(0)}.
 			\end{equation}
\end{enumerate}
\end{corollary}

\begin{proof}
 For the proof we have to combine Definition \ref{def:PolEns} of the P\'olya ensembles, Theorem~\ref{thm:mult-trans-pol}, and the relations
 \begin{equation}\label{p2.1}
 \begin{split}
  \mathcal{F}\left(\left[-\frac{\partial}{\partial x}\omega(x)\right];s_c\right)=&\imath s_c\mathcal{F}\omega(s_c),\\
  \mathcal{M}\left(\left[-x\frac{\partial}{\partial x}\omega(x)\right];s_c\right)=&s_c \mathcal{M}\omega\left(s_c\right),\\
  \mathcal{H}_\nu\left(\left[x^{\nu}\frac{\partial}{\partial x}\frac{1}{x^{\nu-1}}\frac{\partial}{\partial x}\omega(x)\right];s_c\right)=&{-}s_c\mathcal{H}_\nu\omega(s_c),
 \end{split}
 \end{equation}
 for suitable functions $\omega$. These relations can be proven via integrations by parts. For the Hankel transform $\mathcal{H}_\nu$, one needs that the boundary terms vanish, which follows from our integrability assumptions at infinity and from the limit condition in Def.~\ref{def:PolyaEns}\,(3) at the origin.
 
After inserting Eqs.~\eqref{p2.1} into Eqs.~\eqref{S1-pol} -- \eqref{S2-pol},
we may pull those factors out of the determinants in the numerators which are independent of the index~$b$. This leaves us with Vandermonde determinants $\Delta_n(s)$ multiplied by products of the functions $\mathcal{F}\omega(s_j)$, $\mathcal{M}\omega(s_j-(n-1)/2)$,
and $\mathcal{H}_\nu\omega(s_j)$, respectively. The Vandermonde determinants cancel with those in the denominators. 
Since the densities are~normalized, we may use Eq.~\eqref{normalization-def} to fix the constants $C_n[\omega]$. This finishes the proof.
\end{proof}

\begin{remark}\label{rm:3}
Cor.~\ref{cor:mult-trans-pol} may be generalized to non-normalized and even to signed P\'olya ensembles. Here the respective multivariate transform is still proportional to the product of the univariate transforms, and the corresponding normalizing constants may be \emph{determined} by means of the latter formulas in \eqref{S1-pol-der} -- \eqref{S2-pol-der}.
\end{remark}


\section{Proofs of the Main Results}
\label{sec:proofs}

With the help of the multivariate transforms calculated in Theorem~\ref{thm:mult-trans-pol} and Corollary~\ref{cor:mult-trans-pol}, it is easy to prove Theorem~\ref{thm:pol-conv} and Corollary~\ref{cor:pol-conv}:

\begin{proof}[Proof of Theorem~\ref{thm:pol-conv}]
Since the proofs of the three parts are very similar,
we give the proof for part (3) only.
Let $f_M := \PE_M(\omega_1,\hdots,\omega_n)$ and $h_M := \PE_M(\omega)$,
and let $f_A$ and $h_A$ denote the associated induced densities as in \eqref{I-M}.
By $K$-invariance, we may work with the induced transforms and apply Theorem~\ref{thm:conv-fact}:
\begin{align}
  \hankel_\nu(\mathcal{I}_M(f_M \ast h_M))(s)
= \hankel_\nu(f_A \ast_\nu h_A)(s)
= \hankel_\nu f_A(s) \, \hankel_\nu h_A(s) \,.
\end{align}
After using Theorem~\ref{thm:mult-trans-pol} and Corollary~\ref{cor:mult-trans-pol}, we can push the factors from the transform of the P\'olya ensemble into the determinant from the transform of the polynomial ensemble, thereby obtaining
\begin{align}
  \hankel_\nu(\mathcal{I}_M(f_M \ast h_M))(s)
&\propto \frac{\det[\hankel_\nu \omega_b(s_c)]_{b,c=1,\hdots,n}}{\Delta_n(-s)}
 \prod_{j=1}^{n} \hankel_\nu \omega(s_j) \nonumber\\
&= \frac{\det[\hankel_\nu \omega_b(s_c) \, \hankel_\nu \omega(s_c)]_{b,c=1,\hdots,n}}{\Delta_n(-s)} \,.
\end{align}
For the entries of the resulting determinant we can apply Theorem~\ref{thm:conv-fact} for the univariate case $n=1$:
\begin{align}
  \hankel_\nu(\mathcal{I}_M(f_M \ast h_M))(s)
\propto \frac{\det[\hankel_\nu (\omega_b \ast_\nu \omega)(s_c)]_{b,c=1,\hdots,n}}{\Delta_n(-s)} \,.
\end{align}
In~the end we apply the uniqueness theorem by which a $K$-invariant matrix density is determined by its multivariate transform.
\end{proof}

\begin{proof}[Proof of Corollary~\ref{cor:pol-conv}]
The proof works along the same lines as that of Theorem~\ref{thm:pol-conv}, only that we may use Corollary~\ref{cor:mult-trans-pol} for both ensembles and, hence, have no determinant.

Additionally, we need to check that the conditions in Def.~\ref{def:PolyaEns} are preserved under convolution. For brevity, we confine ourselves to a rough outline of the argument for part (3). The integrability and differentiability conditions can be checked using the relations
\begin{equation}
\int_{\myreal_+} (F \ast_\nu G) = \int_{\myreal_+} F \cdot \int_{\myreal_+} G
\end{equation}
and
\begin{equation}
(\partial_x x^{\nu+1} \partial_x x^{-\nu}) (F \ast_\nu G) = (\partial_x x^{\nu+1} \partial_x x^{-\nu} F) \ast_\nu G \,.
\end{equation}
The extra limit condition then follows from the observation that,
given the other conditions,
$\lim_{x \to 0} (x^{\nu+1} \partial_x x^{-\nu} F)(x) = 0$ 
is equivalent to
$\int_0^\infty (\partial_x x^{\nu+1} \partial_x x^{-\nu} F)(x) \, dx$ $= 0$.
\end{proof}

For the proof of  Theorem \ref{thm:rel-Polya}, we need some preparations.
In order to show that an \emph{integrable} function $f$ is a P\'olya frequency function of order $N$, 
it suffices to check the condition~\eqref{Polya-def} for $n = 1$ and $n = N$.

\begin{lemma}[Sufficient Condition for P\'olya Frequency Function of Order $N$] 
\label{lemma:PFF}
Let the function $f : \mathbb{R} \longrightarrow \mathbb{R}$ be integrable
and suppose that the condition~\eqref{Polya-def} holds for $n = 1$ and $n = N$.
Then $f$ is a P\'olya frequency function of order $N$.
\end{lemma}

\begin{proof}
By way of induction, it suffices to show that if a non-negative and integrable function $f$ satisfies condition \eqref{Polya-def} for $n = N$, then it also satisfies condition \eqref{Polya-def} for $n = N-1$.
In doing so, we may assume that $f \ne 0$,
because the claim is trivial otherwise.
Thus, there exists a real number $z_*\in\mathbb{R}$ with $f(z_*) > 0$.

We fix $x_1 < \hdots < x_{N-1}$, $y_1 < \hdots < y_{N-1}$.
Then we know that for \emph{any}  $x_N\geq x_{N-1}$ and $y_N\geq y_{N-1}$,
\begin{align}
\label{eq:PFF-1}
\hat{D} := \det \Big[ f(x_j - y_k) \Big]_{j,k=1,\hdots,N} \ge 0 \,,
\end{align}
and we must show that 
\begin{align}
\label{eq:PFF-2}
\det \Big[ f(x_j - y_k) \Big]_{j,k=1,\hdots,N-1} \ge 0 \,.
\end{align}
To this end, we expand the determinant~\eqref{eq:PFF-1} in the last row and the last column,
\begin{align}
\label{eq:PFF-3}
\hat{D} = f(x_N - y_N) \det A + \sum_{j,k=1}^{N-1} (-1)^{j+k-1} f(x_j - y_N) f(x_N - y_k) \det A^{[j:k]}.
\end{align}
Here, $A$ denotes the $(N-1) \times (N-1)$ matrix in Eq.~\eqref{eq:PFF-2}, and $A^{[j:k]}$ denotes the
$(N-2) \times (N-2)$ matrix obtained from $A$ by removing the $j$'th row and the $k$'th column. Note that the matrix $A$ and, hence, the maximum $K := \max_{j,k=1,\hdots,n} |\det A^{[j:k]}|$ do not depend on $x_N$ and $y_N$.

Now, for sufficiently large $z\in\mathbb{R}$, set $x_N := z + z_*$ and $y_N := z$ with $z_*$ as above.
We define the auxiliary function $h(z) := \sum_{j=1}^{N-1} f(x_j - z) + \sum_{k=1}^{N-1} f(z + z_* - y_k)$. The integrability of $f$ carries over to $h$. Hence, there exist real numbers $z_m> \linebreak[1] \max\{x_{N-1}-z_*,y_{N-1}\}$ such that $ K^{1/2} h(z_m) < 1/m$ for all $m \in \mynat$
and $\lim_{m \to \infty} z_m = \infty$. It then follows from Eq.~\eqref{eq:PFF-3} that
\begin{multline}
\label{eq:PFF-4}
f(z_*) \det A \ge \hat{D} - K\sum_{j,k=1}^{N-1} f(x_j - z_m) f(z_m + z_* - y_k)\\ \ge \hat{D} - K(h(z_m))^2 \geq \hat{D} - \frac{1}{m^2}
\end{multline}
for all $m \in \mynat$. Since $\hat{D} \ge 0$ and $f(z_*) > 0$, this implies $\det A \ge 0$,
and the proof is~complete.
\end{proof}

To the best of our knowledge Lemma~\ref{lemma:PFF} is a new result. It is helpful when checking whether a density is a P\'olya frequency function, especially when proving some statements of our Theorem~\ref{thm:rel-Polya}.

\enlargethispage{1.0\baselineskip}

\subsection{Proof of Theorem \ref{thm:group-pol}}\label{sec:groupintegrals-proof}
Since the proofs of parts (1) -- (3) are very similar, 
we provide the~full details for the most complicated part (3) only,
and confine ourselves to rough sketches for parts (1) and (2).

(3) For any $\varepsilon > 0$, let $q_{M,\varepsilon}$ be the Gaussian density on $M$ as in Example \ref{ex:PolyaEns}\,(a), and set $p_{M,\varepsilon} := p_M \ast q_{M,\varepsilon}$. Then, by basic properties of the convolution, we~have
\begin{align}
\label{p3.1.02a}
\lim_{\epsilon\to0} \int_M \Big| (p_M - p_{M,\epsilon})(\iota_M(y) - \tilde{x}) \Big| \, d\tilde{x} = 0
\end{align}
for any fixed $y \in A$, with $\iota_M(y)$ as in Eq.~\eqref{I-M}. By the change of variables $\tilde{x} \to k \iota_M(x) k^*$, where $k \in K$ and $x \in A$, it follows that
\begin{align}
\label{p3.1.02c}
\lim_{\epsilon\to0} \int_A \int_{K} \Big| (p_M - p_{M,\epsilon})(\iota_M(y)-k\iota_M(x)k^*) \Big|  \, d^*k \, (\det x)^\nu \, |\Delta(x)|^2 \, dx = 0 \,.
\end{align}
Thus, for fixed $y \in A$, since convergence in $L^1$ implies almost sure convergence along some subsequence, we may find a sequence $(\epsilon_m)_{m \in \mynat}$ of positive numbers such~that $\lim_{m \to \infty} \epsilon_m = 0$ and
\begin{align}
\label{p3.1.03}
\lim_{m \to \infty} \int_{K} \Big| (p_M - p_{M,\epsilon_m})(\iota_M(y)-k\iota_M(x)k^*) \Big| \, d^*k = 0
\end{align}
for almost all $x \in A$. 

\pagebreak[2]

We now prove \eqref{group-int1} with $p_M$ replaced by $p_{M,\varepsilon}$.
Fix $\varepsilon > 0$ and $x,y \in A$,
and set $p_A := \mathcal{I} p_M$ and $p_{A,\varepsilon} := \mathcal{I} p_{M,\varepsilon}$.
Then
$(\hankel_{\nu} p_{A,\varepsilon})(s) = (\hankel_{\nu} p_A)(s) \cdot \prod_{j=1}^{n} e^{-\varepsilon s_j}$
for all $s \in \myreal^n$,
where the first factor is bounded and the second factor is integrable 
even after multi\-plication by $|\Delta_n(s)|_2^2$.
Thus, multivariate Fourier inversion (recall from~Eq. \eqref{S2-def}
that the Hankel transform arises from the Fourier transform \eqref{fourier1} 
in the matrix space $M$) yields
\begin{equation}
\label{p3.1.01}
p_{M,\epsilon}(z) = \hat{C}
  \int_{A} (\hankel_{\nu} p_{A,\varepsilon})(s) \left(\int_{K} \exp[-\imath \tr z\tilde{k} \iota_M(s)\tilde{k}^*]d^*\tilde{k}\right) (\det s)^\nu \, \Delta_n^2(s) \, ds \,,
\end{equation}
where $z \in M$ and $\hat{C}$ is the normalizing constant in the inverse Fourier transform.
Replacing $z$ with $\iota_M(y) - k\iota_M(x)k^*$, 
integrating with respect to $k$
and ex\-changing the order of integration, it follows that
\begin{multline}
\label{p3.1.08}
\int_K p_{M,\epsilon}(\iota_M(y) - k\iota_M(x)k^*) \, d^*k
\propto
  \int_{A} (\hankel_{\nu} p_{A,\varepsilon})(s) \\ \times \left(\int_{K}\int_{K} \exp[-\imath \tr ((\iota_M(y) - k\iota_M(x)k^*)\tilde{k} \iota_M(s)\tilde{k}^*)] \, d^*k \, d^*\tilde{k} \right)
(\det s)^\nu \, \Delta_n^2(s) \, ds \,.
\end{multline}
By the invariance of the Haar measure under the translation $k\to \tilde{k} k$,
the inner double integral in \eqref{p3.1.08} factorizes into
\begin{equation}\label{p3.1.2}
\int_{K}\exp[\imath \tr k\iota_M(x)k^*\iota_M(s)]d^*k\int_{K}\exp[-\imath \tr \tilde{k}\iota_M(y)\tilde{k}^*\iota_M(s)]d^*\tilde{k},
\end{equation}
where both integrals are Berezin-Karpelevich integrals~\cite{BK,GW}. 
Hence, using Eq.~\eqref{Berezin-Karpelevich}, it follows that
\begin{multline}\label{p3.1.3}
 \Delta_n(y)\Delta_n(x)\int_{K}p_{M,\varepsilon}(\iota_M(y)-k\iota_M(x)k^*)d^*k
 \propto \int_{A} (\mathcal{H} p_{A,\varepsilon})(s) \\
\,\times\, \det \Big[ \frac{J_{\nu}(2\sqrt{x_bs_c})}{(x_bs_c)^{\nu/2}} \Big]_{b,c=1,\hdots,n} 
 \, \det \Big[ \frac{J_{\nu}(2\sqrt{y_bs_c})}{(y_bs_c)^{\nu/2}} \Big]_{b,c=1,\hdots,n} \, (\det s)^\nu \, ds \,.
\end{multline}
Since the multivariate Hankel transform $\hankel_{\nu} p_{A,\varepsilon}(s)=\prod_{j=1}^n (\hankel_\nu \omega)(s_j) e^{-\varepsilon s_j}$ factorizes by Eqs.~\eqref{S1-pol-der} and \eqref{S2-pol-der},
we can apply Andr\'eief's integration theorem~\cite{Andreief} to~obtain
\begin{multline}\label{p3.1.5}
 \Delta_n(y)\Delta_n(x)\int_{K}p_{M,\varepsilon}(\iota_M(y)-k\iota_M(x)k^*)d^*k\\
 \propto \det\left[\int_0^\infty (\hankel_{\nu} \omega)(s)e^{-\varepsilon s} \, 
\frac{J_{\nu}(2\sqrt{x_cs})}{(x_cs)^{\nu/2}}\frac{J_{\nu}(2\sqrt{y_bs})}{(y_bs)^{\nu/2}}s^\nu\,ds\right]_{b,c=1,\ldots, n} \,.
\end{multline}
Using \eqref{p3.1.5} for $n = 1$, we see that the integral inside the determinant in \eqref{p3.1.5} is proportional to the group integral of the expression $p_{M(1),\varepsilon}(\iota_{M(1)}(y_b)-k\iota_{M(1)}(x_c)k^*)$ over the set $K(1)$. 

The normalizing constant in \eqref{group-int1} results from the normalizing constants in the inverse Fourier transform and the Berezin-Karpelevich integral, and is hence independent of $p_{M,\varepsilon}$. Instead of bookkeeping all constants, it can be determined 
by choosing $p_M(y)\propto\exp[-\tr (y^*y)/2]$ to be Gaussian,
where the group integrals in Eq.~\eqref{group-int1} may be evaluated explicitly
using \eqref{Berezin-Karpelevich}.
This completes the proof of Eq.~\eqref{group-int1} when $p_M$ is replaced with $p_{M,\varepsilon}$.

Finally, for $y$ and $x$ as in \eqref{p3.1.03}, we may deduce Eq.~\eqref{group-int1} by letting $m \to \infty$ in the corresponding result for $p_{M,\varepsilon_m}$. Then the integral of $p_{M,\varepsilon}$ converges to that of $p_{M}$ by \eqref{p3.1.03}, while the integral of $p_{M(1),\varepsilon}$ converges to that of $p_{M(1)}$ due to the continuity of $p_{M(1)}(s)$ for $n > 1$. (For $n = 1$, Eq.~\eqref{group-int1} is trivial.)

(1) This proof is almost the same. We must only replace $\iota_M$ with the identity,
use the Harish-Chandra-Itzykson-Zuber integral \eqref{Harish-Chan} instead of the Berezin-Karpelevich integral \eqref{Berezin-Karpelevich} and note that the integral over $K_2(1) = {\rm U}(1)$ reduces to a constant by commutativity.

(2) This proof goes along the same ideas as for part (3).
First of all, we argue that it is  sufficient to prove \eqref{group-int3}
for $p_{G,\varepsilon} := p_G \circledast q_\varepsilon$ instead of $p_G$,
where $q_\varepsilon$ is the multivariate `log-normal' density on $G$
as in Example \ref{ex:PolyaEns}\,(a).
Then we insert the spherical inversion formula
\cite[Thm. IV.7.5]{Helgason3}%
\begin{equation}
p_{G,\varepsilon}(g) \propto \int_{\myreal^n} (\mys p_{G,\varepsilon})(n \eins+\imath s) \, \varphi(gg^*,-n \eins-\imath s) \, |\Delta_n(\imath s)|^2 \, ds
\end{equation}
with $\eins := (1,\hdots,1) \in \myreal^n$
on the left side in Eq. \eqref{group-int3}
and interchange the group integral with the integral over $s$. 
The spherical function
\begin{equation}
\varphi(gg^*,s) := \int_{{\rm{U}}(n)}  \prod_{j=1}^n \det (\Pi_{j,n} kgg^*k^* \Pi_{j,n}^*)^{s_j-s_{j+1}-1} d^*k
\end{equation}
(where $s_{n+1} := - \tfrac{n+1}{2}$)
satisfies the functional equation \cite[Prop. IV.2.2]{Helgason3}
\begin{equation}\label{p3.3.1}
 \int_{{\rm{U}}(n)} \varphi(x^{-1/2}kyk^*x^{-1/2},s)d^*k=\varphi(x^{-1},s)\varphi(y,s).
\end{equation}
Then we use the explicit representation~\eqref{Gelfand-Naimark} of the spherical function  due to Gelfand and Na\u{\i}mark~\cite{GelNai}. The Vandermonde determinants cancel out and, since the spherical transform $\mathcal{S} p_{G,\varepsilon}$ is of the form~\eqref{S3-pol-der}, we can again apply Andr\'eief's integration theorem~\cite{Andreief}. Setting $n=1$, the entries of the resulting determinant can be recognized as inverse Mellin transforms, yielding Eq.~\eqref{group-int3}.
\qed

\begin{remark}\label{rm:1}
The identities in Thm.~\ref{thm:group-pol} do not rely on the positivity of the weights, as follows from the preceding proofs. Thus,
they are valid for signed densities as well.
\end{remark}

\subsection{Proof of Theorems~\ref{thm:rel-Polya} and~\ref{thm:rel-Polya.b}}\label{sec:thmproof}

For $n = 1$, the claims are trivial, since both the non-negativity of the density \eqref{jpdf-pol-ens} and the P\'olya frequency property reduce to the non-negativity of the function $\omega$. Hence, we~assume that $n > 1$.

\begin{proof}[Proof of Theorem~\ref{thm:rel-Polya}]\

(1) Let $\omega\in L^{1}_{H_2}(\mathbb{R})$ define
 a P\'olya ensemble on $H_2$, and without loss of generality suppose $\omega$ to be normalized, i.e. $\int_{-\infty}^\infty\omega(x')dx'=1$. Then we know that
\begin{equation}\label{proof4.1.1}
p_{H_2(1)}(x)=\omega(x)\geq0\quad{\rm and}\quad p_{H_2}(y)\geq0
\end{equation}
for all $x\in\mathbb{R}$ and $y\in H_2$.  Due to the group integral~\eqref{group-int4}, we also have
\begin{multline}\label{proof4.1.2}
 \Delta_n(x)\Delta_n(y)\det\left[\omega(y_b-x_c)\right]_{b,c=1,\ldots,n} \\ 
 = n! \, C_{H_2} \, \Delta_n^2(y)\Delta_n^2(x) \int_{K_2}p_{H_2}(y-kxk^*) \, d^*k \geq 0
\end{multline}
for almost all $y,x\in D$. By continuity, the non-negativity extends to all $y,x \in D$.
Combining these two non-negativity properties for $N=1$ and $N=n$ with the integrability of $\omega$,
we know from Lemma~\ref{lemma:PFF} that $\omega$ is a P\'olya frequency function of order $n$.

Conversely, let $\omega\in L^{1}_{H_2}(\mathbb{R})$ be a P\'olya frequency function of order $n$ with $\omega\neq0$. Then we know, compare e.g.\@ Theorem~2 in Ref.\cite{Karlin:1957}, that
\begin{equation}\label{proof4.1.3}
\Delta_n(x)\det\left[(-\partial_{x_c})^{b-1}\omega(x_c)\right]_{b,c=1,\ldots,n}\geq0
\end{equation}
for all $x\in D$. We still have to check that the normalizing constant $C_n[\omega]$ does not vanish. But this follows from Eq.~\eqref{S1-pol-der} and the observation that $(\fourier\omega)(0) > 0$ because $\omega$ is non-negative and does not vanish identically; see also Remark \ref{rm:3}.
Thus, $\omega$ gives rise to a P\'olya ensemble on $H_2$.

(2) We pursue the same ideas as in the proof of part (1). Consider a P\'olya en\-semble $p_{G}$ on $G$ associated with the normalized weight $\omega\in L_{G}^1(\mathbb{R}_+)$. The counter\-parts of Eqs.~\eqref{proof4.1.1} and \eqref{proof4.1.2} are
\begin{equation}\label{proof4.2.1}
p_{G(1)}(x)=\omega(x)\geq0\quad{\rm and}\quad p_{G}(g)\geq0
\end{equation}
for all $x\in\mathbb{R}_+$ and $g\in G$ and
\begin{multline}\label{proof4.2.2}
 \Delta_n(x)\Delta_n(y)\det\left[\omega\left(\frac{y_b}{x_c}\right)\right]_{b,c=1,\ldots,n} \\ = n! \, C_G \left(\prod_{j=1}^{n} \mellin \omega(j) \right) \frac{\Delta_n^2(y)\Delta_n^2(x)}{\det (x^{n-1})}\int_{\rm U(n)} p_{G}(x^{-1/2}ky^{1/2}) \, d^*k \geq 0
\end{multline}
for almost all $y,x\in A$. 
Again by continuity, the non-negativity extends to all $x,y \in A$. Setting $\tilde\omega(x)=\omega(e^{-x})e^{-x}$, 
we find that $\tilde\omega\ne 0$ is integrable, non-negative and satisfies
\begin{multline}\label{proof4.2.3}
\Delta_n(x) \Delta_n(y) \det\Big[\tilde\omega\left(y_b-x_c\right)\Big]_{b,c=1,\ldots,n} \\[+2pt] = 
\Delta_n(x) \Delta_n(y) \det\Big[\omega\left(e^{x_c-y_b}\right)e^{x_c-y_b}\Big]_{b,c=1,\ldots,n}\geq0
\end{multline}
for all $x,y\in D$,
since the products $\Delta_n(x) \Delta_n(y)$ and $\Delta_n(e^{-x}) \Delta_n(e^{-y})$
have the same sign.
Therefore, $\tilde\omega(x)$ is a P\'olya frequency function of order $n$ by Lemma~\ref{lemma:PFF}.

Conversely, when we start from a P\'olya frequency function $\tilde\omega(x)=\omega(e^{-x})e^{-x}$ of order $n$ with $\omega\in L_{G}^1(\mathbb{R}_+)$ and $\omega\neq0$, it follows similarly as in Eq.~\eqref{proof4.1.3} that
\begin{multline}\label{proof4.2.4}
\Delta_n(y)\det y\det\left[(-y_c\partial_{y_c})^{b-1}\omega(y_c)\right]_{b,c=1,\ldots,n} \\[+2pt]
= \Delta_n(-e^{-x})\det\left[(-\partial_{x_c})^{b-1}\tilde\omega(x_c)\right]_{b,c=1,\ldots,n} \geq 0
\end{multline}
for all $y=e^{-x}\in A$, since $\Delta_n(-e^{-x})$ has the same sign as $\Delta_n(x)$.
This time the normalizing constant $C_n[\omega]$ is given by Eq.~\eqref{S3-pol-der}; see Remark \ref{rm:3}. The positivity and integrability conditions on $\omega$ immediately tell~us that $\mathcal{M}\omega(j)>0$ for all $j=1,\ldots,n$, so that $C_n[\omega]\neq0$. 
This implies that $\omega$ gives rise to a P\'olya ensemble on~$G$.
\end{proof}

\begin{proof}[Proof of Theorem~\ref{thm:rel-Polya.b}]\
Let $\tilde\omega\in L^{1}_{H_2}(\mathbb{R})$ with support contained in $[0,\infty[$, and let $\omega$ be defined as in Eq.~\eqref{polya-rel.1}. The weight $\tilde\omega$ is an $L^{1}$-function on $\mathbb{R}_+$ and for any $x > 0$, the function $\tilde\omega(y)\exp[-x/y]x^\nu/y^{\nu+1}$ as well as its derivatives $\partial_x^l [ \tilde\omega(y)\exp[-x/y]x^\nu/y^{\nu+1} ]$, $l=1,\ldots,2n-2$, are integrable in $y$. Thus, it is easy to see that the derivatives of $\omega$ up to order $2n-2$ exist. Moreover, $\omega$ is positive  because the integrand is positive.

In the next step we check that the integrability conditions in the definition of the set $L^{1}_M(\mathbb{R}_+)$, see Def.~\ref{def:PolyaEns}\,(c). For this purpose we will repeatedly use the identity
\begin{equation}\label{p4.2.1}
 \left(x^\nu\frac{\partial }{\partial x}x^{1-\nu}\frac{\partial }{\partial x}\right)^l\omega(x)=\frac{1}{\Gamma[\nu+1]}\int_0^\infty \left(\frac{x}{y}\right)^\nu\exp\left[-\frac{x}{y}\right]\bigg[\left(-\frac{\partial }{\partial y}\right)^l\tilde\omega(y)\bigg]\frac{dy}{y}
\end{equation}
for $l=0,\ldots,n-1$, which can be proven via  integration by parts. We also recall the operator identity $x^\nu\partial_x x^{1-\nu}\partial_x=\partial_x x^{\nu+1}\partial_x x^{-\nu}$ and that the boundary terms vanish because of the definition of the set $ L^{1}_{H_2}(\mathbb{R})$, see Def.~\ref{def:PolyaEns}\,(a), i.e. $\lim_{y\to0}\tilde\omega^{(l)}(y)=0$ for all $l=0,\ldots,n-2$. Using~\eqref{p4.2.1}, we have
\begin{align}\label{p4.2.2}
&\int_0^\infty \left|x^{\kappa-1}\left(x^\nu\frac{\partial }{\partial x}x^{1-\nu}\frac{\partial }{\partial x}\right)^l\omega(x)\right|dx \nonumber\\
\leq&\frac{1}{\Gamma[\nu+1]}\int_0^\infty\int_0^\infty \left|x^{\kappa-1}\left(\frac{x}{y}\right)^\nu\exp\left[-\frac{x}{y}\right]\bigg[\left(-\frac{\partial }{\partial y}\right)^l\tilde\omega(y)\bigg]\right|\frac{dy}{y}\,dx \nonumber\\
=&\frac{\Gamma[\nu+\kappa]}{\Gamma[\nu+1]}\int_0^\infty \left|y^{\kappa-1}\bigg[\left(-\frac{\partial }{\partial y}\right)^l\tilde\omega(y)\bigg]\right|dy < \infty
\end{align}
for any $\kappa\in[1,n]$ and $l=0,\ldots,n-1$.

Now we check whether the boundary conditions in the definition of $L^{1}_{M}(\mathbb{R}_+)$ are satisfied. 
For this purpose we let $\tilde\omega^{(l)}$ denote the $l$-th derivative of $\tilde\omega$. 
We choose an auxiliary parameter $\gamma\in{]2/3,1[}$. Then, for $x > 0$ small enough, 
using Eq. \eqref{p4.2.1} in the first step, we have the estimate
\begin{align}
 &\left|x^{\nu+1}\frac{\partial}{\partial x}\frac{1}{x^\nu}\left(\frac{\partial}{\partial x}x^{\nu+1}\frac{\partial}{\partial x}\frac{1}{x^\nu}\right)^l\omega(x)\right|\nonumber\\
 =&\frac{1}{\Gamma[\nu+1]}\left|\int_0^\infty \left(\frac{x}{y}\right)^{\nu+1}\exp\left[-\frac{x}{y}\right]\tilde\omega^{(l)}(y)\frac{dy}{y}\right|\nonumber\\
 =&\frac{1}{\Gamma[\nu+1]}\left|\int_0^\infty y^{\nu}\exp\left[-y\right]\tilde\omega^{(l)}\left(\frac{x}{y}\right)dy\right|\nonumber\\
 \leq&\frac{1}{\Gamma[\nu+1]}\biggl(\left|\int_0^{x^\gamma} y^{\nu+2}\exp\left[-y\right]\tilde\omega^{(l)}\left(\frac{x}{y}\right)\frac{dy}{y^2}\right|+\left|\int_{x^\gamma}^\infty y^{\nu}\exp\left[-y\right]\tilde\omega^{(l)}\left(\frac{x}{y}\right)dy\right|\biggl)\nonumber\\
 \leq&\frac{x^{\gamma(\nu+2)}e^{-x^\gamma}}{\Gamma[\nu+1]}\left|\int_0^{x^\gamma} \tilde\omega^{(l)}\left(\frac{x}{y}\right)\frac{dy}{y^2}\right|+\max_{\lambda\in[0,x^{1-\gamma}]}|\tilde\omega^{(l)}(\lambda)|\left|\int_{x^\gamma}^\infty \frac{y^{\nu}\exp\left[-y\right]}{\Gamma[\nu+1]}dy\right|\nonumber\\
 \leq&\frac{x^{\gamma(\nu+2)-1}}{\Gamma[\nu+1]}\int_0^\infty |\tilde\omega^{(l)}(y)|dy+\max_{\lambda\in[0,x^{1-\gamma}]}|\tilde\omega^{(l)}(\lambda)|.\label{4p.2.3}
\end{align}
Now let $x \to 0$. The first term vanishes because $|\tilde\omega^{(l)}|$ is integrable for any $l=0,\ldots,n-1$ and $\gamma>2/3\geq1/(\nu+2)$ for any $\nu\geq-1/2$. The second term vanishes because  $\lim_{\lambda\to0}\tilde\omega^{(l)}(\lambda)=0$ for any $l=0,\ldots,n-2$. Thus we have
\begin{equation}\label{4p.2.4}
 \lim_{x\to0}\left|x^{\nu+1}\frac{\partial}{\partial x}\frac{1}{x^\nu}\left(\frac{\partial}{\partial x}x^{\nu+1}\frac{\partial}{\partial x}\frac{1}{x^\nu}\right)^l\omega(x)\right|=0
\end{equation}
for all $l=0,\ldots,n-2$.

At last, we show that the density associated with $\omega$ is indeed non-negative.
To this end, we establish a relation between the densities on $A$ 
corresponding to $\omega$~and~$\tilde\omega$. For $x \in A$ with $x_1,\hdots,x_n$ pairwise different, we have 
\begin{align}\label{4p.2.5}
0 &\le \int_A \left(\int_{K_n}\exp[-\tr xk a^{-1}k^*]d^*k\right) \frac{\det x^\nu}{\det a^{\nu+n}}\Delta_n(a)\det\left[\left(-\frac{\partial }{\partial a_c}\right)^{b-1}\tilde\omega(a_c)\right]_{b,c=1,\ldots,n}da \nonumber\\
&=\frac{\prod_{j=0}^{n-1}j!}{\Delta_n(x)}\int_A \det[e^{-x_b/a_c}]_{b,c=1,\ldots,n} \left(\frac{\det x}{\det a}\right)^\nu\det\left[\left(-\frac{\partial }{\partial a_c}\right)^{b-1}\tilde\omega(a_c)\right]_{b,c=1,\ldots,n}\frac{da}{\det a} \nonumber\\
&=\frac{\prod_{j=1}^{n}j!}{\Delta_n(x)}\det\left[\int_0^\infty \left(\frac{x_c}{y}\right)^\nu\exp\left[-\frac{x_c}{y}\right]\left[\left(-\frac{\partial }{\partial y}\right)^{b-1}\tilde\omega(y)\right]\frac{dy}{y}\right]_{b,c=1,\ldots,n} \nonumber\\
&=\frac{\prod_{j=1}^{n}(j!\Gamma[\nu+1])}{\Delta_n(x)}\det\left[\left(x_c^\nu\frac{\partial }{\partial x_c}x_c^{1-\nu}\frac{\partial }{\partial x_c}\right)^{b-1}\omega(x_c)\right]_{b,c=1,\ldots,n}.
\end{align}
The initial inequality holds because the density associated with $\tilde\omega$ is non-negative and the other factors are positive.
In the next steps, we use the Harish-Chandra-Itzykson-Zuber integral~\eqref{Harish-Chan} in combination with the relation $\Delta_n(-a^{-1}) = \linebreak \det (a^{1-n}) \, \Delta_n(a)$, which shifts the exponent of the determinant $\det a$,
the Andr\'eief identity~\cite{Andreief}, and Eq.~\eqref{p4.2.1}, 
respectively.

Apart from the normalizing constant, the last line in \eqref{4p.2.5} is exactly the joint probability density 
of the squared singular values of a  P\'olya ensemble on $M_\nu$ divided by $|\Delta_n(x)|^2$. 
Finally, the relation $\hankel_\nu \omega(0) = \int_0^\infty \tilde\omega(x) \, dx > 0$ shows
that the corresponding normalizing constant $C_n[\omega]$ is indeed positive; see also Rem.~\ref{rm:3}.
This completes the~proof of Theorem~\ref{thm:rel-Polya.b}.
\end{proof}

\section{Conclusions and Outlook}
\label{sec:conclusio}

In the present work we have extended the notion of a \emph{polynomial ensemble of derivative type} to the classes of complex rectangular matrices, Hermitian anti-symmetric matrices and Hermitian anti-self-dual matrices. We have shown that these ensembles have nice closure properties with respect to additive convolution
(Theorem \ref{thm:pol-conv} and Corollary \ref{cor:pol-conv}),
thereby extending previous results for complex square matrices \cite{KK:2016a} and complex Hermitian matrices \cite{KR:2016}. In fact, by using the appropriate multivariate transforms from harmonic analysis, all these classes of matrices may be handled in a unified way.

Furthermore, for each of these classes of matrices, we have addressed the question which weight functions give rise to polynomial ensembles of derivative type. We have shown 
in Theorems \ref{thm:rel-Polya} and \ref{thm:rel-Polya.b} that these weight functions are closely related to P\'olya frequency functions \cite{Polya:1913,Polya:1915,Schoen:1951}. For this reason, we propose the shorter name \emph{P\'olya ensembles} for the resulting ensembles. 

Another main result are the group integrals in Theorem \ref{thm:group-pol} 
which generalize known group integrals~\cite{GR:1989,O:2004,HO:2006} 
and which seem to be of independent interest. 
Moreover, on the practical side, the relation to P\'olya frequency functions yields a~multitude of examples of P\'olya ensembles and associated group integral identities. Some of them are highly non-trivial and go beyond the classical results in random matrix theory.

As is evident from our proofs, both the convolution theorems and the group integral identities are intimately related to the fact that the relevant multivariate transform \emph{factorizes}.
Furthermore, we would like to emphasize that these results do not really require the positivity of the matrix densities. They also hold for signed matrix densities as long as the integrability and differentiability conditions on the underlying weight function are satisfied. The question is whether these results can be extended even beyond these conditions. A natural candidate are the signed measures such that the relevant multivariate transforms factorizes, but a closer description seems to require the language of distributions.

Interestingly, the cases of real anti-symmetric matrices of odd dimension and of Hermitian anti-self-dual matrices lead to exactly the same results. From a group theoretical perspective, this has to be expected since the roots are the same apart from their length. Thus the random matrix ensembles show the same eigenvalue statistics, but the eigenvector statistics should be different.

Finally, it is crucial for our results that the group integrals involved in the multivariate transforms admit explicit expressions and that the ensembles correspond to the Dyson index $\beta=2$. How our results can be extended to other symmetry classes of matrices, e.g. real-symmetric matrices or Hermitian self-dual matrices, is still an~open problem. 

\medskip


\section*{Acknowledgements}

We want to thank Gernot Akemann, Friedrich G\"otze and Arno Kuijlaars
for fruitful discussions on this topic. Moreover we acknowledge support 
by {\it CRC 701 ``Spectral Structures and Topological Methods in Mathematics''} 
as well as by grant {\it AK35/2-1 ``Products of Random Matrices''}, 
both funded by {\it Deutsche Forschungsgemeinschaft (DFG)}. 
The work of Yanik-Pascal F\"orster was formerly supported by \emph{Studien\-stiftung des deutschen Volkes}.

\makeatletter
\@setaddresses
\let\addresses\empty
\makeatother

\end{document}